\documentclass[10pt]{article}
\usepackage{amsfonts,amsmath,amsthm,amssymb,color}
\usepackage{fullpage}

\numberwithin{equation}{section}

\newtheorem{thm}{Theorem}[section]
\newtheorem*{thm*}{Theorem}
\newtheorem{prop}[thm]{Proposition}
\newtheorem{question}[thm]{Question}
\newtheorem*{question*}{Question}

\newtheorem{cor}[thm]{Corollary}

\newtheorem{defin}[thm]{Definition}

\newtheorem{lemma}[thm]{Lemma}

\newtheorem{remark}[thm]{Remark}
\newtheorem*{remark*}{Remark}

\newcommand{\ip}[1]{\langle #1 \rangle}

\long\def\symbolfootnote[#1]#2{\begingroup%
\def\thefootnote{\fnsymbol{footnote}}\footnote[#1]{#2}\endgroup}

\begin{document}

\title{Concerning the existence of Einstein and Ricci soliton metrics on solvable Lie groups}
\author{M.Jablonski} \date{}
\maketitle

    \begin{abstract} In this work we investigate solvable and nilpotent Lie groups with special metrics.  The metrics of interest are left-invariant Einstein and algebraic Ricci soliton metrics.  Our main result shows that the existence of a such a metric is intrinsic to the underlying Lie algebra.  More precisely, we show how one may determine the existence of such a metric by analyzing algebraic properties of the Lie algebra in question and infinitesimal deformations of any initial metric.  
    
    Our second main result concerns the isometry groups of such distinguished metrics.  Among the completely solvable unimodular Lie groups (this includes nilpotent groups), if the Lie group admits such a metric, we show that the isometry group of this special metric is maximal among all isometry groups of left-invariant metrics.  We finish with a similar result for locally left-invariant metrics on compact nilmanifolds.\end{abstract}

\section{Introduction}

Our primary interest in this work is (left-invariant) Einstein metrics on non-compact Lie groups.  All known examples of such metrics occur on solvable Lie groups. In fact, all known examples of non-compact homogeneous Einstein metrics are isometric to solvable Lie groups with left-invariant metrics; this is the content of the well-known Alekseevskii conjecture which has been verified in dimensions 4 and 5 \cite{Jensen:HomogEinsteinSpacesofDim4,Nikonorov:NoncompactHomogEinstein5manifolds}.  We restrict ourselves to this class of Lie groups and ask when such a group admits an Einstein metric.

The answer in the compact setting is well-known.  If a compact group $G$ admits an Einstein metric, then either
    \begin{quote}
        (i) $G$ is a torus (zero scalar curvature) or\\
        (ii) $G$ is a compact semi-simple Lie group (positive scalar curvature).
    \end{quote}
The first case of Ricci flat follows from a general result of Alekseevskii-Kimelfeld where it is shown that any homogeneous Ricci flat space is actually flat  \cite{AlekseevskiiKimelfeld:StructureOfHomogRiemSpacesWithZeroRicciCurv}.  For the positive scalar curvature  case, such metrics are characterized as critical points of the total scalar curvature functional \cite{Jensen}.  On compact semi-simple groups, Einstein metrics are not unique and most groups admit at least 2 such metrics (this is in sharp contrast to the solvable setting, where Einstein metrics are unique up to isometry and scaling).

We observe that the existence of an Einstein metric on a compact Lie group is completely determined by the underlying Lie algebra.  The flat case corresponds to abelian Lie algebras.  The positive case is characterized as follows.  Let $\mathfrak g$ denote the Lie algebra of $G$, then $G$ is compact semi-simple if and only if the Killing form $B(X,Y) = tr (ad\ X \circ ad\ Y)$, with $X,Y\in \mathfrak g$,  is negative definite.

Our work is motivated by, and seeks to answer, the following questions.

\begin{question}\label{question: when does Einstein or solsol exist}Given a solvable Lie group, how can one determine if it admits an Einstein or solsoliton metric?
\end{question}

\begin{question}\label{question: how to find Einstein or solsol}If a solvable Lie group is known to admit an Einstein or solsoliton metric, how does one find it?
\end{question}

As the curvature tensors of left-invariant metrics are left-invariant, the above questions reduce to studying inner products on a given Lie algebra.  More precisely, consider a solvable Lie algebra $\mathfrak g$ with corresponding Lie group $G$.  Let $\ip{\cdot, \cdot }$ be an inner product on $\mathfrak g$ with corresponding left-invariant metric on $G$.  The Ricci curvature of $(G,\ip{\cdot,\cdot})$ is completely determined by its values on $\mathfrak g$ (by left-invariance) and is given by the formula
    $$ric(X,X)= -\frac{1}{2}\sum_i | [X,X_i] |^2 -\frac{1}{2}\sum_i \ip{[X,[X,X_i]],X_i} +\frac{1}{4}\sum_{i,j} \ip{[X_i,X_j],X}^2 -\ip{[Z,X],X}   $$
where $\{X_i\}$ is an orthonormal basis of $\mathfrak g$ and $Z\in \mathfrak g$ is the unique vector satisfying $\ip{Z,X}=tr(ad\ X)$.
Observe that the Ricci curvature is completely determined by the Lie bracket and inner product on $\mathfrak g$.  Denoting the $(1,1)$-Ricci tensor by $Ric$, Question \ref{question: when does Einstein or solsol exist} may be rephrased as follows.
\begin{question*}
    When  does there exist an inner product $\ip{\cdot, \cdot}$ on $\mathfrak g$, such that
    \begin{equation}\label{eqn: definition of solsoliton}
        Ric = cId + D
    \end{equation}
    for some $c\in \mathbb R$ and some $D\in Der(\mathfrak g)$?
\end{question*}
Here $Der(\mathfrak g)$ denotes the algebra of derivations of $\mathfrak g$.  Terminology: when $D=0$, the metric is called an Einstein metric;  when $D\not = 0$, the metric is  called a solsoliton.  Solsolitons are algebraic examples of Ricci solitons, see Section \ref{sec: Riem Lie groups}.

Our main result shows that a definite answer to Question \ref{question: when does Einstein or solsol exist} can be obtained by analyzing only local data: algebraic information about the underlying Lie algebra and infinitesimal deformations of any metric; see Section \ref{sec: alg for Einstein metric} for complete details and the procedure referenced by the following theorem.

\bigskip
\noindent
\textbf{Theorem \ref{thm: existence of Einstein of solvable}.} \textit{ Let $G$ be a solvable Lie group with Lie algebra $\mathfrak g$.  The existence of a left-invariant Einstein metric on $G$ can be determined by analyzing the following: 1) adjoint action of $\mathfrak g$ on itself, 2) the commutator subalgebra $\mathfrak n = [\mathfrak g,\mathfrak g]$, 
and 3) infinitesimal deformations of any initial metric on $\mathfrak n$.}

\begin{remark*} The existence of an Einstein metric on a solvable Lie group is now a local question. Similarly, one can formulate the question of existence of a solsoliton in terms of local data.
\end{remark*}

In general, the existence of an Einstein or Ricci soliton metric is not a local question.  It might appear at first glance that the existence of left-invariant Einstein metrics on Lie groups is a local question since the verification of Equation \ref{eqn: definition of solsoliton} uses only the inner product and Lie bracket on $\mathfrak g$.  However, asking if a Lie group admits such a metric amounts to asking if there exists a zero of the function $|| Ric_g - \frac{sc(g)}{n}Id||^2$ on the open set of inner products.  It is not clear if this is a local question for non-solvable Lie groups; e.g., there does not exist a solution when the Lie algebra is $\mathfrak{sl}_2\mathbb R$.

In the setting of compact homogeneous spaces $G/H$, the Einstein question has received a great deal of attention and there are some partial results on the existence of such metrics.  For example, if $G$ is a compact semi-simple group and $H$ is a maximal connected subgroup of $G$, then $G/H$ admits a $G$-invariant Einstein metric \cite{Wang-Ziller:ExistenceAndNonexistenceOfHomogEinstein}.  However, there exist many examples of homogeneous spaces $G/H$ which don't admit $G$-invariant Einstein metrics.  Presently, there are not any general, local conditions that guarantee/exclude the existence of such metrics on compact homogeneous spaces; see \cite{Bohm-Wang-Ziller:AVariationalApproachforCompactHomogEinsteinMflds} for the current state of research.

\bigskip

Our work builds on the strong structural results of Heber \cite{Heber} and Lauret \cite{LauretStandard}.  These works take the  first step in reducing the problem on the solvable group to a smaller solvable group, a one dimensional extension of a nilpotent group.  This smaller solvable Lie group admits an Einstein metric if and only if its nilradical admits a so-called nilsoliton metric and the underlying Lie algebra is the extension of the nilradical by a so-called pre-Einstein derivation.  Reducing the problem to analyzing the nilradical is an algebra problem.

To study the nilradical we build on the work of Nikolayevsky \cite{Nikolayevsky:EinsteinSolvmanifoldsandPreEinsteinDerivation}.  Using a combination of measuring algebraic information and infinitesimal deformations of metrics on the nilradical, we translate the Einstein problem into a local problem.  (While we could skip this analysis on the nilradical and couple our techniques directly with the work of Heber \cite[Section 6]{Heber}, we present our results in the given framework as these methods extend directly to solsoliton and nilsoliton metrics.  See Section \ref{sec: alg for nilsoliton} for more details.)

\bigskip

Our second main result is an algebraic decomposition theorem for solvable groups admitting Einstein metrics.  If one were to classify the solvable groups admitting Einstein or solsoliton metrics,  one would want to construct such groups from basic building blocks.
The question of existence of an Einstein or solsoliton metric can be reduced to the case that the underlying Lie algebra is indecomposable.

\bigskip

\noindent
\textbf{Theorem \ref{thm: solv admit solsoliton reduces to irreducible components}} \textit{Let $G$ be a solvable Lie group whose Lie algebra $\mathfrak g = \mathfrak g_1 + \mathfrak g_2$ is a direct sum of ideals.  Then $G$ admits a non-flat solsoliton, resp. flat, metric if and only if both $G_1$ and $G_2$ admit non-flat solsoliton, resp. flat, metrics.}

\bigskip

\noindent
\textbf{Corollary \ref{cor: solv admit solsoliton reduces to irreducible components}} \textit{Let $G$ be a solvable Lie group whose Lie algebra $\mathfrak g = \mathfrak g_1 + \mathfrak g_2$ is a direct sum of ideals.  Then $G$ admits an Einstein metric if and only if both $G_1$ and $G_2$ admit Einstein metrics of the same sign.}

\begin{remark*} A similar decomposition result has appeared for nilsolitons and nilpotent Lie groups, see \cite{Nikolayevsky:EinsteinSolvmanifoldsandPreEinsteinDerivation} and \cite{Jablo:FinitenessTheorem-compatibleSubgroups}.
To our knowledge, the above algebraic decomposition theorem is the first of its kind for homogeneous Einstein spaces.  It would be interesting to know if there is a similar theorem in the compact setting; there are some partial results of B\"ohm in this direction \cite[Theorem B]{Bohm:HomogEinsteinMetricsAndSimplicialComplexes}.
\end{remark*}

In addition to providing a local formulation of the existence of such a metric on a solvable Lie group, we demonstrate how to recover such metrics by following two natural curves of metrics (see Proposition \ref{prop: finding Einstein via 2 curves}).  Using these curves, we demonstrate that solsolitons (when they exist) are the most symmetric metric on  completely solvable unimodular Lie groups (this class includes nilpotent Lie groups).

\bigskip
\noindent
\textbf{Theorem \ref{thm: comp solv unimod solsoliton have max isom grp}} \textit{Let $S$ be a completely solvable unimodular Lie group that admits a solsoliton metric.  Let $g$ be any left-invariant metric.  Then there exists a left-invariant soliton metric $g'$ such that $Isom(S,g) \subset Isom(S,g')$, as groups.}\bigskip

This result is extended to compact nilmanifolds with \textit{local nilsoliton metrics} in Theorem \ref{thm: isom grp of local nilsoliton is maximal}.

\bigskip

\textit{Table of Contents.} Section \ref{sec: Riem Lie groups} reviews information on Lie groups with left-invariant metrics.  Section \ref{sec: variety of Lie brackets} discusses the space of Lie brackets, moment maps, and distinguished orbits.  Section \ref{sec: solitons vs distinguished orbits} compares the existence of solsoliton metrics and distinguished orbits. Section \ref{sec: bracket flow} studies the bracket flow with applications to finding solsoliton metrics and comparisons of  isometry groups.  Section \ref{sec: compact quotients} states results on compact nilmanifolds. Section \ref{sec: stratifying V} discusses the stratification of the space of Lie brackets. Section \ref{sec: pre-einstein derivations} covers pre-Einstein derivations and the previous work of Nikolayevsky.  Sections \ref{sec: alg for nilsoliton} and \ref{sec: alg for Einstein metric} show the existence of nilsoliton and Einstein metrics are intrinsic to the underlying Lie algebra, respectively.

\section{Riemannian Lie groups}\label{sec: Riem Lie groups}
A Lie group $G$ is called a \textit{Riemannian Lie group} if it is endowed with a left-invariant metric.
The following question motivates much of our work.
    \begin{question} Among left-invariant metrics on a given Lie group, is there a canonical or preferred choice of metric?
    \end{question}
Special metrics are often characterized as those having good curvature properties or as solutions to some extremal problem.  For example, we are interested in metrics satisfying one or several of the following conditions
    \begin{quote} 1. Nice curvature properties or curvature tensor\\
        2.  Critical points of a Riemannian functional on the set of metrics\\
        3.  More generally, fixed points of a dynamical system\\
        4.  Large  group of isometries
    \end{quote}

Classically, spaces with constant sectional and Ricci curvature have been investigated as preferred metrics on a manifold; the later are known as \textit{Einstein metrics}.  We are interested in left-invariant Einstein metrics when they exist; however, many of our Lie groups are not able to admit left-invariant Einstein metrics.  For example, any non-abelian nilpotent Lie group cannot admit an Einstein metric \cite{Jensen:HomogEinsteinSpacesofDim4,Milnor:LeftInvMetricsonLieGroups} and many solvable Lie groups do not admit Einstein metrics (cf. Theorem \ref{thm: classification of solv admitting negative Einstein}).

Given that many of our Lie groups will not be able to admit an Einstein metric, we explore alternate metrics in search of one which is `distinguished' or preferred in some way.  One natural generalization of the Einstein metric is the so-called \textit{Ricci soliton metric}.  A metric $g$ is called a Ricci soliton if there exists a complete vector field $X\in \mathfrak X(G)$ and constant $c\in \mathbb R$ such that
    $$ric_g = cg + \mathcal L_X g$$
where $\mathcal L_X$ is the Lie derivative generated by $X$.  These metrics arise as special solutions to the Ricci flow which are of the form $g(t) = c(t) \varphi^*(t)g$ where $c(t)$ is a real-valued function and $\varphi(t)$ is the 1-parameter group of diffeomorphisms that generates $X\in \mathfrak X(G)$.  Hence, Ricci solitons  can be viewed as generalized fixed points of the Ricci flow on the space of metrics modulo diffeomorphisms; see \cite{ChowKnopf} for an introduction to Ricci flow and Ricci solitons.

On a Riemannian Lie group $G$ there is a very natural kind of Ricci soliton, which we call an algebraic Ricci soliton, or algebraic soliton for short.  A left-invariant metric $g$ is called an \textit{algebraic soliton} if the Ricci tensor (evaluated at the identity $e\in G$) is of the form
    $$Ric = cId +D$$
for some $D\in Der(\mathfrak g)$.  In the above equation, we have written the Ricci $(1,1)$-tensor $Ric$ instead of the $(2,0)$-tensor $ric$ as in the definition of Ricci soliton above; this is done for ease of presentation.  The derivation $D$ generates the 1-parameter family of automorphisms $\varphi(t)=exp(tD)$ which is the corresponding family of diffeomorphisms from the definition of Ricci soliton.

\begin{remark*}  1) By left-invariance of the metric on $G$, it suffices to only prescribed the value of $Ric$ at $e\in G$. 2) Presently, the only known examples of Ricci solitons on Riemannian Lie groups are algebraic solitons and these are only known to exist on solvable Lie groups.
\end{remark*}

When $G$ is a nilpotent group, such a metric is often called a \textit{nilsoliton} in the literature and if $G$ is solvable, such a metric has been called a \textit{solsoliton}.

\subsection*{Nilpotent Lie groups}

As nilpotent Lie groups cannot admit Einstein metrics, we will search for left-invariant Ricci soliton metrics on such spaces.  These metrics satisfy several of the criteria listed above for being a preferred metric.  Before stating the next theorem, we need some notation.  Given a nilpotent Lie group $G$, denote by $\mathcal M^G$ the set of left $G$-invariant metrics on $G$ (i.e. inner products on $\mathfrak g$).  Recall  that $Ric_h$ is left $G$-invariant for $h\in \mathcal M^G$.

\begin{thm}\cite{LauretNilsoliton,Jablo:RiemFunctionalOnNilpotent}
    \begin{enumerate}
    \item  (Algebro-analytic characterization)\\
    Every left-invariant Ricci soliton on a nilpotent Lie group $G$ is an algebraic soliton, that is,
        $$Ric = cId + D$$
        for some $D\in Der(\mathfrak g)$.
    \item (Variational characterization)\\
        A metric $g\in \mathcal M^G$ is a nilsoliton if and only if $g$ is a critical point of the functional
    $$F(h) = \frac{tr\ Ric_h^2}{sc(h)^2}$$
on $\mathcal M^G$, where $sc(h)$ denotes the (constant) scalar curvature of $(G,h)$. (This functional makes sense as a real-valued function by left-invariance of $Ric$.)
    \item (Uniqueness)\\
    If a nilsoliton exists on $G$, then it is unique (up to isometry) after prescribing the scalar curvature.
    \end{enumerate}
\end{thm}

\begin{remark}The above results are primarily due to Lauret \cite{LauretNilsoliton}.  Part \textit{ii.} was originally proven where the functional considered was on the space of Lie brackets, this result has been extended to the setting of left-invariant Riemannian metrics in \cite{Jablo:RiemFunctionalOnNilpotent} where new convergence results are also obtained.\\
\indent
There is a small gap in the original proof of Part \textit{i.} above.  In that work, it is shown that if the metric is a soliton, then there exists a derivation $D\in Der(\mathfrak g)$ such that $Ric = cId + \frac{1}{2}(D+D^t)$.  The gap is fixable by showing that there exists such $D$ satisfying $D^t\in Der(\mathfrak g)$; this will appear in a future work of that author \cite{Lauret:personalcommunication}.
\end{remark}

Observe that $g\in \mathcal M^G$ is a nilsoliton if and only if $g$ is a critical point of the functional
    $$F(h) = tr\ Ric_h^2$$
along the set $\{ h\in \mathcal M^G\ | \ sc(h)=sc(g) \}$.  A similar functional can be studied on compact nilmanifolds, see Section \ref{sec: compact quotients}.\bigskip

In the sequel, we will see that nilsolitons have maximal isometry groups among all left-invariant metrics.  More precisely we have the following result, see Corollary \ref{cor: nilsoliton have max isom grp}.

\bigskip

\noindent \textbf{Corollary} \textit{Let $G$ be a nilpotent Lie group which admits a nilsoliton metric.  Let $g\in \mathcal M^G$ be any left-invariant metric on $G$.  Then there exists a nilsoliton $g'\in \mathcal M^G$ on $G$ such that $Isom (g) \subset Isom(g')$.}

\begin{remark} In this way, we see that nilsolitons are the most symmetric left-invariant metric on a nilpotent Lie group.  However, there are other non-soliton metrics which can have the same isometry group.  It would be interesting to know which other geometric properties these highly symmetric nilmanifolds share with nilsolitons.
\end{remark}

The property of being an Einstein nilradical is intrinsic to the underlying Lie algebra.  Moreover, one can verify this via an algorithm, see Theorem \ref{thm: existence of nilsoliton} and Section \ref{sec: alg for nilsoliton}.

\bigskip

\noindent
\textbf{Theorem} \textit{ Let $N$ be a nilpotent Lie group with Lie algebra $\mathfrak n$.
 The existence of a nilsoliton metric on a nilpotent Lie group $N$ 
 can be determined by analyzing  the derivation algebra $Der(\mathfrak n)$ and infinitesimal deformations of any initial metric on $\mathfrak n$.}

\subsection*{Solvable Lie groups}

The analysis of Riemannian solvable Lie groups splits into two distinct sets: unimodular and non-unimodular groups.  A Lie group $G$ is called \textit{unimodular} if $| det( Ad(g) )| =1 $ for all $g\in G$.  Notice, in particular, that $tr\ ad(X) = 0$ for $X\in \mathfrak g$ when $G$ is unimodular.  If $G$ is not unimodular it is called \textit{non-unimodular}.

Within both these classes of solvable Lie groups, we are interested in those which are completely solvable.  A solvable group $G$ is called \textit{completely solvable} if $ad(X):\mathfrak g \to \mathfrak g$ has only real eigenvalues, for all $X\in \mathfrak g$.  Observe that nilpotent Lie groups are unimodular completely solvable, with eigenvalues all zero.\bigskip

As (non-abelian) solvable Lie groups are non-compact, any left-invariant Einstein metric on such a group must have scalar curvature less than or equal to zero by Bonnet-Myers theorem.  We recall the following result (cf. \cite{Dotti:RicciCurvUnimodularSolv,AlekseevskiiKimelfeld:StructureOfHomogRiemSpacesWithZeroRicciCurv}).

\begin{thm}[Alekseevskii-Kimel'fel'd, Dotti]Let $G$ be a solvable Lie group with left-invariant Einstein metric.
    \begin{quote} 1.  The scalar curvature is negative if and only if $G$ is non-unimodular, and\\
                2.  The scalar curvature is zero if and only if $G$ is unimodular.
    \end{quote}
In the case of Ricci flat, the metric is actually flat (i.e., constant zero sectional curvature).
\end{thm}

In \cite{AlekseevskiiKimelfeld:StructureOfHomogRiemSpacesWithZeroRicciCurv} it is actually shown that any homogeneous Ricci flat space must be flat.  Lie groups with flat left-invariant metrics are necessarily unimodular solvable and have been classified \cite{Milnor:LeftInvMetricsonLieGroups}.

\begin{thm}[Milnor]\label{thm: Milnor flat}  A Riemannian Lie group $G$ is flat if and only
if its Lie algebra $\mathfrak g$ (with inner product) splits as an orthogonal direct sum $\mathfrak g = \mathfrak a \oplus \mathfrak n$ where $\mathfrak n$ is an abelian ideal (the nilradical) and $\mathfrak a$ is an abelian Lie algebra such that $ad\ X$ is skew-symmetric for $X\in \mathfrak a$.  Such $G$ is necessarily solvable.
\end{thm}

Given the simple algebraic structure of these solvable Lie groups, one may classify the solvable Lie groups admitting flat metrics.  These are the groups whose Lie algebras are constructed as follows  (cf. Theorem \ref{thm: classification of solv admitting negative Einstein} where the negative Einstein case is considered).

\begin{prop}\label{prop: classification of solv admitting flat} Let $\mathfrak n$ be an abelian Lie algebra and $\mathfrak a \subset Der(\mathfrak n)$ an abelian, reductive subalgebra of derivations, all of whose elements have purely imaginary eigenvalues.  If $G$ is  a (solvable) Lie group whose  Lie algebra is the semi-direct product $\mathfrak a \oplus \mathfrak n$, then $G$  admits a flat metric. Conversely, every such solvable group arises this way.
\end{prop}

\begin{proof} Picking a basis of $\mathfrak n$, we may identify it with $\mathbb R^n$.  Via this identification, the  abelian, reductive algebra $\mathfrak a \subset \mathfrak {gl}(n,\mathbb R)$.

It is well-known that there exists an inner product on $\mathbb R^n$ such that $\mathfrak a$ is stable under the transpose operation, see \cite{Mostow:SelfAdjointGroups}.  
As the eigenvalues of every element in $\mathfrak a$ are purely imaginary, we see that $\mathfrak a$ consists of skew-symmetric derivations.  Now Milnor's theorem above applies.
\end{proof}

As the unimodular solvable Lie groups admitting Einstein metrics are understood, our attention is dedicated to analyzing the non-unimodular solvable groups admitting Einstein metrics  and both kinds of solvable groups which admit solsoliton metrics.
As in the case of nilpotent Lie groups, solsolitons (including Einstein metrics) have several rigid properties.  The following results may be found in \cite{Lauret:SolSolitons}.

\begin{thm}[Lauret]\label{thm: lauret solsolitons criteria} Let $(G,g)$ be a solvable Riemannian Lie group with metric Lie algebra $(\mathfrak g, g)$. Let $\mathfrak n$ be the nilradical of $\mathfrak g$ (with induced metric) and $\mathfrak a = \mathfrak n ^\perp$, so that $\mathfrak g=\mathfrak a\oplus \mathfrak n$.
    \begin{enumerate}
        \item (Structural results and the standard property)\\
            The Riemannian Lie group $G$ is a solsoliton (i.e. $Ric=cId + D$) if and only if\\
            a) $(\mathfrak n, g)$ (with the induced metric) is a nilsoliton\\
            b) $\mathfrak a = \mathfrak n^\perp$ is abelian\\
            c) $(ad\ A)^t \in Der(\mathfrak g)$ (or equivalently, $[ad\ A, (ad\ A)^t] = 0$) for all $A \in \mathfrak a$.\\
            d) $g(A,A) = \frac{-1}{c} tr\ S(ad\ A)^2$ for all $A\in \mathfrak a$, where $S(ad\ A) = \frac{1}{2}(ad\ A+(ad\ A)^t)$
        \item (Solsolitons are not shrinkers)\\
            The constant $c$ satisfying $Ric=cId +D$ satisfies $c\leq 0$.  Moreover, if $c=0$ then $D=0$ and the metric is flat (cf. Theorem \ref{thm: Milnor flat}).  This says solsolitons are either so-called steady or expanding Ricci solitons (as opposed to shrinking solitons).
        \item (Uniqueness)\\
            If a solvable Lie group admits a solsoliton, then it is unique (up to isometry) after prescribing the scalar curvature.
    \end{enumerate}
\end{thm}

Observe that Part \textit{i.d)} implies that the eigenvalues of $ad\ A$ are not all purely imaginary, for any $A\in \mathfrak a$.

\begin{remark} 1) There does exist a variational characterization for Einstein solvmanifolds with codimension 1 nilradical which realizes these spaces as critical points of a `modified scalar curvature function' (see \cite{Lauret:StandardEinsteinSolvAsCriticalPoints}).
2) As in the case of flat Einstein metrics,  we have a characterization of solvable Lie groups which admit Einstein and solsoliton metrics,  see Theorem \ref{thm: classification of solv admitting negative Einstein} and Corollary \ref{cor: classification of solv admitting negative Einstein}.
\end{remark}

Let $G$ be a solvable unimodular Lie group.  If $G$ admits a non-trivial solsoliton, it cannot admit a (flat) Einstein metric, and conversely, if $G$ admits a (flat) Einstein metric, it cannot admit a non-trivial solsoliton.  In this way, solsolitons are a preferred metric on unimodular solvable groups that cannot admit (flat) Einstein metrics.  This preference is defended by the following, see Theorem \ref{thm: comp solv unimod solsoliton have max isom grp}.

\begin{thm*}
Let $G$ be unimodular completely solvable Lie group which admits a solsoliton metric.  Given any left-invariant metric $g$, there exists a solsoliton  metric $g'$ such that $Isom(g) \subset Isom(g')$.
\end{thm*}

\begin{remark}Presently, we do not have such a theorem when the group is not completely solvable or for nonunimodular solvable Lie groups.  Our techniques do allow one to embed a large portion of the isometry group of any metric into the Einstein or solsoliton metric, however they do not allow one to embed the entire isometry group.  This question will be addressed in future work.
\end{remark}

\subsection*{Isometry groups}
The group of isometries of a Riemannian solvable Lie group is particularly simple when the group in question is a completely solvable unimodular group.  The following is Theorem 4.3 of \cite{GordonWilson:IsomGrpsOfRiemSolv}.

\begin{thm}[Gordon-Wilson]\label{thm: isom of compl solv unimod} Let $G$ be a completely solvable unimodular Lie group with left-invariant metric $g$.  The full isometry group is a semi-direct product
    $$Isom(G,g) = K \ltimes G$$
where $K \subset Aut(G)$ is the isotropy subgroup of $Isom(G,g)$ preserving the identity $e\in G$.  Under the natural identification $Aut(G) \simeq Aut(\mathfrak g)$ we have $K \simeq Aut(\mathfrak g)\cap O(g)$, where $O(g)$ is the orthogonal group of the inner product $g$ on $\mathfrak g$.
\end{thm}

Observe that this theorem covers the case of any Riemannian nilpotent Lie group.

\begin{defin}\label{def: alg isom group} Let $(G,g)$ be a Riemannian Lie group.  The group $G\rtimes (Isom (g) \cap Aut(G))$ is a subgroup of isometries which we call the algebraic isometry group.
\end{defin}

The above theorem says that the algebraic isometry group of a completely unimodular solvable group is the whole isometry group.  For non-unimodular solvable groups, it is well-known that the full isometry group is significantly larger than its algebraic isometry group \cite{GordonWilson:IsomGrpsOfRiemSolv}.

\section{The Variety of Lie Brackets}\label{sec: variety of Lie brackets}
A Lie group $G$ with a left-invariant metric $\ip{\ , \ }$ gives rise to a \textit{metric Lie algebra} $\{ \mathfrak g, \ip{\ , \ } \}$, where $\mathfrak g$ is the Lie algebra of $G$ and the inner product on $\mathfrak g$ is the restriction of the left-invariant metric to $T_eG \simeq \mathfrak g$.  Conversely, a metric Lie algebra gives a left-invariant metric on any Lie group with said Lie algebra.  We are primarily interested in simply-connected Lie groups.

We say that two metric Lie algebras $\{\mathfrak g_1, \ip{\ , \ }_1\}$ and $\{\mathfrak g_2, \ip{\ , \ }_2\}$ are \textit{isomorphic} if there exists a Lie algebra isomorphism $\phi : \mathfrak g_1 \to \mathfrak g_2$ such that $\ip{\ ,\ }_1 = \phi^*\ip{\ ,\ }_2$.  Such an isomorphism lifts to give an isometry between the simply-connected Riemannian Lie groups $\{G_1, \ip{\ , \ }\} \to \{G_2,\ip{\ , \ } \}$.  In general, most isometries do not arise this way, however, for nilpotent and some solvable groups, this is how all isometries arise (see Theorem \ref{thm: isom of compl solv unimod}).

To obtain good information on Riemannian Lie groups, we study metric Lie algebras by considering a metric Lie algebra as a collection of three objects: a vector space $\mathbb R^n$, a Lie bracket $[\cdot , \cdot ]$ and an inner product $\ip{\cdot , \cdot }$.  We use what is becoming a standard technique and convert our questions into a frame work that can exploit deep theorems from Geometric Invariant Theory:  Instead of fixing a Lie algebra and varying an inner product on it, we  choose to fix an inner product and vary the underlying Lie algebra structure (within the same isomorphism class).

For $g\in GL(n,\mathbb R)$, we may consider a different (and isomorphic) Lie bracket $g^*[\cdot , \cdot ] = g[g^{-1}\cdot , g^{-1}\cdot]$ and the inner product $g^*\ip{\cdot, \cdot} = \ip{g^{-1}\cdot, g^{-1}\cdot}$.  The following are isomorphic metric Lie algebras
    $$\{\mathbb R^n, g^*[\cdot,\cdot ],\ip{\cdot,\cdot}\} \simeq  \{\mathbb R^n, [\cdot,\cdot],(g^{-1})^*\ip{\cdot,\cdot}\}   $$
via the isomorphism $g^{-1} :\mathbb R^n\to \mathbb R^n$.

We now fix an inner product (the usual one) on $\mathbb R^n$ and study the collection $g^*[\cdot,\cdot ]$, $g\in GL(n,\mathbb R)$.  It is helpful to study not just this collection of isomorphic Lie algebras on $\mathbb R^n$, but instead to study all Lie algebra structures on $\mathbb R^n$.  Consider the vector space $V = \wedge^2(\mathbb R^n)^*\otimes \mathbb R^n$, the space of anti-symmetric, bilinear maps from $\mathbb R^n \times \mathbb R^n \to \mathbb R^n$.  This vector space is endowed with a natural  $GL(n,\mathbb R)$ action:
    $$(g^* [\cdot,\cdot])(v,w) = g[g^{-1}v,g^{-1}w]$$
for $g\in GL(n,\mathbb R)$, $v,w\in \mathbb R^n$.  Via differentiation, we also have an action of $\mathfrak{gl}(n,\mathbb R)$ on $V$.  Given $X\in \mathfrak{gl}(n,\mathbb R)$ and $v,w\in \mathbb R^n$, we have $(X\cdot[\cdot,\cdot])(v,w) = X[v,w]-[Xv,w]-[v,Xw]$.\bigskip

The points of $V$ can be thought of as anti-symmetric algebra structures on $\mathbb R^n$, and two algebra structures are isomorphic if and only if they lie in the same $GL(n,\mathbb R)$-orbit.  Any Lie bracket $[\cdot,\cdot]$ on $\mathbb R^n$ can be realized as a point in $V$ and the subset
    $$\mathcal V = \{ \mu \in V \ | \ \mu \mbox{ satisfies the Jacobi identity } \}$$
is a variety in $V$ whose points are the Lie brackets on $\mathbb R^n$.  Additionally, we will restrict our attention to some interesting subsets of $\mathcal V$: let $\mathcal N$ denote the Lie brackets which are nilpotent, $\mathcal S$ denote the Lie brackets which are solvable, and $\mathcal{CS}$ denote the Lie brackets which are completely solvable (cf. Section \ref{sec: Riem Lie groups}).  We have the following containments
    $$\mathcal N \subset \mathcal{CS} \subset \mathcal S$$
These subsets are all closed in $V$.  We will often abuse language and refer to $\mu \in \mathcal V$ as a Lie algebra, when we really mean the pair $\{\mathbb R^n,\mu\}$.

Given a Lie bracket $\mu \in \mathcal V$, we will denote by $\mathfrak s_\mu$ the metric Lie algebra $\{\mathbb R^n, [\cdot,\cdot ] ,\ip{\cdot,\cdot}\}$; the corresponding simply connected Lie group with left-invariant metric will be denoted by $S_\mu$.

\begin{remark} For $k\in O(n,\mathbb R)$, the groups $S_\mu$ and $S_{k\cdot \mu}$ are isometric.  However, in general, one often has $S_\mu$ and $S_\lambda$ which are isometric but $\lambda \not \in O(n,\mathbb R)\cdot \mu$.
\end{remark}

As it will be of interest later, we point out that the stabilizers of the actions of $GL_n\mathbb R$ and $\mathfrak{gl}_n\mathbb R$ have relevant meaning: $(GL_n\mathbb R)_\mu =  Aut(\mu)$ and $(\mathfrak{gl}_n\mathbb R)_\mu = Der(\mu)$.

\subsection*{The moment map and geometry of orbits}
The geometry of $GL(n,\mathbb R)$-orbits in $\mathcal V$ is intimately connected to algebraic properties of the Lie group $S_\mu$ associated to $\mu$ and geometric structures that the group can admit.  For example, consider the induced action of $SL(n,\mathbb R)$ on $V$. For $\mu \in \mathcal V$, $SL(n,\mathbb R)\cdot \mu$ is closed in $V$ if and only if $\mu$ is a semi-simple Lie algebra, see  \cite{Lauret:MomentMapVarietyLieAlgebras}.

When $\mu$ is nilpotent, it is known that so-called \textit{distinguished orbits}  (which are generalizations of closed orbits, see Definition \ref{def: disting orbits})  correspond precisely to nilpotent Lie groups which admit left-invariant Ricci soliton metrics, see Theorem \ref{thm: nilsoliton vs disting point}.  In the sequel, we show that distinguished orbits also play an role in the study of solvable Lie groups and solsolitons.\bigskip

Before defining distinguished orbits, we must define the moment map of a representation of a reductive group.  The moment map defined here works for non-compact groups and is a natural extension of the usual one defined for compact groups, cf. \cite{EberleinJablo}.

The group $GL(n,\mathbb R)$ is endowed with the Cartan involution $\theta (g) = (g^t)^{-1}$, where $* ^t$ denotes the transpose operation.  By differentiating, we have an involution on $\mathfrak{gl}(n,\mathbb R)$ as well, which we denote by the same symbol: $\theta (X) = -X^t$.

These involutions gives rise to so-called Cartan decompositions
    $$GL(n,\mathbb R) = KP    \quad  \quad  \mathfrak{gl}(n,\mathbb R) = \mathfrak k \oplus \mathfrak p$$
where $K = O(n) = \{g\in GL(n,\mathbb R) \ | \ \theta(g)=g\}$, $P = \{ g\in GL(n,\mathbb R) \ | \ \theta(g) = g^{-1} \}$, $\mathfrak k = Lie K = \mathfrak{so}(n) = \{X\in \mathfrak{gl}(n,\mathbb R) \ | \ \theta(X) = X \}$, and $\mathfrak p = symm(n)= \{X\in \mathfrak{gl}(n,\mathbb R) \ | \ \theta(X)=-X\}$.  Here $symm(n)$ denotes the symmetric $n\times n$ matrices.  Additionally, $P= exp(\mathfrak p)$, where $exp$ is the Lie group exponential.

Let $G$ be a real algebraic reductive subgroup of $GL(n,\mathbb R)$ which is $\theta$-stable.    For such groups, we obtain a Cartan decomposition $G=K_GP_G$ where $K_G = K\cap G =G^\theta = \{g\in G \ | \ \theta(g)=g \}$ is a maximal compact subgroup of $G$ and $P_G=P\cap G = \{g\in G \ | \ \theta(g) = g^{-1} \}$.  Similarly, the Lie algebra $\mathfrak g = Lie \ G$ has a Cartan decomposition $\mathfrak g = \mathfrak k_G \oplus \mathfrak p_G$.  Often we will drop the subscript $G$ when the group is understood.\bigskip

The (usual) inner product on $\mathbb R^n$ extends of an $O(n)$-invariant inner product $\ip{\ , \ }$ on the vector space $V= \wedge ^2(\mathbb R^n)^* \otimes \mathbb R^n$ as follows
    $$\ip{\lambda , \mu } = \sum_{i < j} \ip{\lambda(e_i,e_j) , \mu(e_i,e_j)} = \sum_{i<j} \ip{\lambda(e_i,e_j) , e_k} \ip{\mu(e_i,e_j),e_k}$$
where $\{e_i\}$ denotes the standard orthonormal basis of $\mathbb R^n$.  If we denote by $\pi$ the representations of $GL(n,\mathbb R)$ and $\mathfrak{gl}(n,\mathbb R)$ on $V$, then it is immediate to see that $\pi(g^t)=\pi(g)^t$ and $\pi(X^t)=\pi(X)^t$, for $g\in GL(n,\mathbb R)$ and $X\in \mathfrak{gl}(n,\mathbb R)$, where the right-hand side represents the metric adjoint with respect to the inner product $\ip{\ ,\ }$ on $V$.

On $\mathfrak{gl}_n\mathbb R$, we consider the standard inner product
    $$\ip{\alpha,\beta} = tr \ \alpha \beta^t = \sum \ip{\alpha e_i,\beta e_i}$$
where $t$ denotes transpose,  $tr$ is the trace, and $\{e_i\}$ is the standard orthonormal  basis of $\mathbb R^n$.  This inner product is $Ad(K)$-invariant and  $ad(\alpha^t) = (ad(\alpha))^t$ for $\alpha \in \mathfrak{gl}_n\mathbb R$. Observe that this inner product restricts to any $\theta$-stable subalgebra $\mathfrak g$.

Given the above choices of a Cartan involution on $\mathfrak g$, and inner products on $\mathfrak g$ and $V$, we may define the \textit{moment map} $m_G:V\backslash \{0\} \to \mathfrak p$ for the action of $G$ on $V$.  This function is defined implicitly by
    $$\ip{m_G(v),\alpha} = \frac{1}{||v||^2} \ip{\pi(\alpha)v,v}, \quad \forall \ \alpha \in \mathfrak p, v\in V$$
Observe that $m_G$ is fixed under rescaling in $V$ and is $K$-equivariant; that is, for $c\in \mathbb R$, $m(cv)=m(v)$ and for $k\in K_G$, $m_G(k\cdot v) = Ad(k) m_G(v)$.  When the group $G$ is understood, we will simply write $m = m_G$.

Using the inner product on $\mathfrak g$, we consider the norm squared of the moment map
    $$F(v) = ||m||^2 : V\backslash \{0\} \to \mathbb R$$
Notice that this function is invariant under rescaling in $V$ and so it may be viewed as a function on spheres in $V$ or on projective space $\mathbb PV$.  The critical points of this function have been extensively studied so as to develop a good understanding of the so-called `nullcone' of complex representations, see \cite{Kirwan,Ness,Marian}.  Moment maps have also been used to study general representations of non-compact real reductive groups to study the geometry of orbits, see e.g. \cite{RichSlow,EberleinJablo,JabloDistinguishedOrbits}.

\begin{defin}\label{def: disting orbits} An orbit $G\cdot v \subset V$ is called distinguished if it contains a critical point of the function $F=||m||^2$.
\end{defin}

Observe that $v\in V$ is a critical point of $F$ if and only if $\pi (m(v)) (v) = rv$ for some $r\in \mathbb R$.  It is a fact that any closed orbit is distinguished with critical value $0$ and so these orbits are a natural generalization of closed orbits, see \cite{JabloDistinguishedOrbits}.  The following theorem motivates a deeper study of $||m||^2$ on $\mathcal N \subset V$.

\begin{thm}[Lauret]\label{thm: nilsoliton vs disting point} Let $N_\mu$ denote the simply connected nilpotent Lie group with left-invariant metric whose Lie algebra $\mathfrak n_\mu$ (with inner product) corresponds to the point $\mu \in \mathcal N$.  Then $N_\mu$ is a nilsoliton if and only if $\mu$ is a critical point of $F(v)=||m||^2(v)$.  Equivalently, $N_\mu$ is an Einstein nilradical if and only if the orbit $GL_n\mathbb R \cdot \mu$ is distinguished.
\end{thm}

In this way, we convert our questions of left-invariant metrics on Lie groups into questions about the geometry of orbits in the space $V$.  By analyzing the geometry of orbits, we obtain our algorithm that determines when a given nilpotent Lie group is an Einstein nilradical, see Section \ref{sec: alg for nilsoliton}.

The above theorem can be found in \cite{Lauret:EinsteinSolvandNilsolitonsCordobaConf2007}.  The last equivalence is not stated using the label of distinguished orbit but is stated using the idea.  In Section 10 of \cite{Lauret:EinsteinSolvandNilsolitonsCordobaConf2007} there are several open questions of interest which are presented.  We state Question \# 5 from this list.

\begin{question}\label{question: gradient flow of einstein nilradical}  Consider the function $F: V \to \mathbb R$ defined by $F(v) = ||m(v)||^2$ where  $m$ is the moment map, as above.  Define $\mu_t$ to be the integral curve of $-grad\ F$ starting at $\mu_0$ on the sphere of radius 2.  Is $\mu_\infty$ (the limit point along the integral curve) contained in the orbit $GL_n\mathbb R \cdot \mu_0$ if $N_{\mu_0}$ is an Einstein nilradical?
\end{question}

In Theorem \ref{thm: einstein nilradical isomorphic to limit under gradient flow} we obtain an affirmative answer to this question.  This result first appeared more generally in \cite{JabloDistinguishedOrbits}, where this was shown to be true for distinguished orbits in any real reductive  algebraic representation.

\begin{remark}\label{rem: geometric interp of moment map}Geometrically, the moment map can be understood as follows.  When $\mu$ is a nilpotent Lie algebra, $m(\mu) = 4 Ric (N_\mu)$.  Moregenerally, if $\mu$ is any Lie algebra with corresponding Lie group $S_\mu$, then $m(\mu) = 4 R$ where $R$ is the tensor appearing in the formula
	$$Ric = R-\frac{1}{2}B-S(ad\ H)$$
here $Ric$ is the Ricci tensor of $S_\mu$, $B$ is the Killing form of the Lie algebra $\mu$ and $S(ad\ H) = \frac{1}{2}(ad\ H + ad\ H^t)$, where $ad\ H$ is a mean curvature vector.  See \cite[Corollary 7.38]{Besse:EinsteinMflds} and \cite[Section 4]{Lauret:SolSolitons} for more details.
\end{remark}

\section{Soliton metrics and distinguished orbits}\label{sec: solitons vs distinguished orbits}

\subsection*{Soliton metrics on nilpotent Lie groups}

In the above section, we stated the relationship between nilsoliton metrics and critical points of the function $F=||m||^2$: they are precisely the same thing, see Theorem \ref{thm: nilsoliton vs disting point}.  As such, we are motivated to study the negative gradient flow of $F$.

\begin{defin}[The bracket flow]\label{def: bracket flow} Let $\mu_t \subset V$ denote the negative gradient flow of $F$ starting at $\mu_0\in V$.
\end{defin}

In Proposition \ref{prop: unique limit point}, it will be seen that the limit of this flow is unique.  We denote this limit point by  $\mu_\infty$.

\begin{thm}\label{thm: einstein nilradical isomorphic to limit under gradient flow} Let $N_{\mu_0}$ be an Einstein nilradical.  Let $\mu_\infty$ denote the limit point of the negative gradient flow of the function $F$ starting at $\mu_0$.  Then $\mu_\infty$ is contained in the orbit $GL_n\mathbb R \cdot \mu_0$; i.e. $N_{\mu_0}$ and $N_{\mu_\infty}$ are isomorphic Lie groups.
\end{thm}

A more general result of this type is true for distinguished orbits and is useful for studying the geometry of solvable Lie groups, see Section \ref{sec: bracket flow}. 
In the setting of nilpotent Lie algebras, the proof can be shortened dramatically by employing the stratification results of Lauret (see Theorem \ref{thm: stratification of V}).  In addition, we obtain some new and interesting geometric results on isometry and automorphism groups of nilpotent Lie groups using the techniques from this proof, see Corollary \ref{cor: nilsoliton have max isom grp} and Proposition \ref{prop: Aut at soliton}.\bigskip

The above relationship between Einstein nilradicals and distinguished orbits has been studied extensively in the literature, see e.g. \cite{LauretNilsoliton,Will:Rank1EinsteinSolvOfDim7,
Payne:ExistenceofSolitononNil,LauretWill:EinsteinSolvExistandNonexist,
Eberlein:prescribedRicciTensor,JabloDistinguishedOrbits,
Nikolayevsky:EinsteinSolvmanifoldsandPreEinsteinDerivation,Nikolayevsky:EinsteinSolvmanifoldsattchedtoTwostepNilradicals,
Nikolayevsky:EinsteinSolvwithSimpleEinsteinDerivation,Nikolayevsky:EinsteinSolvwithFreeNilradical,
Will:CurveOfNonEinsteinNilradicals,Jablo:ModuliOfEinsteinAndNoneinstein}.  Motivated by this, we explore the relationship between distinguished orbits and soliton metrics on solvable groups.

\begin{defin}A Riemannian Lie group $S_\mu$ is said to have a distinguished metric if $\mu$ is a critical point of $F=||m||^2$ for the action of $GL(n,\mathbb R)$ on $V$ (defined above).
\end{defin}

For a geometric interpretation of the moment map, see Remark \ref{rem: geometric interp of moment map}.

\subsection*{Einstein and soliton metrics on solvable Lie groups}

\begin{prop}\label{prop: admit solsoliton implies distinguished orbit} If $S_\mu$ is a solvable group admitting an Einstein or solsoliton metric, then $GL(n,\mathbb R)\cdot \mu$ is a distinguished orbit.
\end{prop}

\begin{proof} To prove this, one consults the work \cite{Lauret:MomentMapVarietyLieAlgebras} where complex Lie algebras are studied. All the results of that paper remain true for real Lie algebras with the Hermitian transpose replaced with the usual transpose.  For a detailed proof of this fact, see \cite{JabloDistinguishedOrbits}.  We warn the reader that the moment map defined there is a multiple of the moment map defined here.  If $n$ denotes the moment map from \cite{Lauret:MomentMapVarietyLieAlgebras} and $m$ denotes the moment map used in this work, then $n=2m$.  Our choice of moment map $m$ is consistent with \cite{Lauret:EinsteinSolvandNilsolitonsCordobaConf2007,Lauret:SolSolitons}

\bigskip

Case 1: $S_\mu$ admits a flat metric.  If $\mu$ corresponds to the flat metric, then $\mu$ is also a critical point of $F =||m||^2$, see \cite[Theorem 4.7]{Lauret:MomentMapVarietyLieAlgebras}.

\bigskip

Case 2: $S_\mu$ admits a non-flat solsoliton.  We only prove this in the case that the nilradical of $\mathfrak s_\mu$ is non-abelian.  The abelian case is similar and we leave it to the diligent reader.

The proof of this case is just a careful comparison of \cite[Theorem 4.8]{Lauret:SolSolitons} with \cite[Theorem 4.7]{Lauret:MomentMapVarietyLieAlgebras}.  The soliton metric and the distinguished metric differ only in their values on $\mathfrak a \times \mathfrak a$, where $\mathfrak a = \mathfrak n^\perp$.  If the nilradical (which is a nilsoliton in either case) satisfies $Ric_\mathfrak n = cId +D$, for some $D\in Der(\mathfrak n)$, and has $sc=-1/4$, then the solsoliton metric on $\mathfrak a$ is
    $$\ip{A,A} = \frac{-1}{c} tr\ S(ad\ A)^2$$
where $S(ad\ A)$ is the symmetric part of $ad\ A$, while the distinguished metric on $\mathfrak a$ is
    $$\ip{\ip{A,A}}= \frac{1}{2}\cdot \frac{-1}{c}tr(\ ad\ A\ (ad\ A)^t)$$
In \cite{Lauret:MomentMapVarietyLieAlgebras}, $\mathfrak a$ is viewed as a subset of $Der(\mathfrak n)$ with $A \simeq ad\ A$.
\end{proof}

\begin{remark}\label{rmk: comp solv unimod - dist metric vs solsoliton and isometry groups}
Observe that when $S_\mu$ is completely solvable,  a stronger statement is true.  In this case,
    $$\ip{\ip{\ , \ }} = \frac{1}{2} \ip{\ , \ } \quad \quad \mbox{on } \mathfrak a \times \mathfrak a$$
and the algebraic isometry groups (cf. Definition \ref{def: alg isom group}) are equal: $Aut(\mu) \cap O(\ip{\ip{\ , \ }}) = Aut(\mu) \cap O(\ip{\ , \ })$.
\end{remark}

\begin{thm}\label{thm: solsolitons and distinguished orbits} Let $S_\mu$ be a completely solvable group.  Then $S_\mu$ admits a solsoliton if and only if $GL(n,\mathbb R)\cdot \mu$ is a distinguished orbit.  Moreover, there is a curve of metrics between the distinguished metric and the solsoliton metric which preserves their algebraic isometry groups.

In particular, when  $S_\mu$ is unimodular, we see that these two Riemannian Lie groups have the same isometry groups.
\end{thm}

The claims in the first paragraph follow from the above observations.  The last claim will be proved in Theorem \ref{thm: comp solv unimod solsoliton have max isom grp}.

\begin{remark} There are solvable groups which admit a distinguished metric, but cannot admit a solsoliton.  For example, if $\mathfrak n$ is a non-abelian Einstein nilradical and $\mathfrak a \subset Der(\mathfrak n)$ is an abelian subalgebra of skew-symmetric endomorphisms, then $S_\mu$ with $\mathfrak s_\mu = \mathfrak a \ltimes \mathfrak n$ cannot admit a solsoliton but does admit a distinguished metric.  See Theorem \ref{thm: classification of solv admitting negative Einstein} and \cite[Theorem 4.7]{Lauret:MomentMapVarietyLieAlgebras}.
\end{remark}

The following has been shown for nilpotent groups in \cite{Nikolayevsky:EinsteinSolvmanifoldsandPreEinsteinDerivation} and \cite{Jablo:FinitenessTheorem-compatibleSubgroups}, but has not appeared in the literature for solvable groups in general.  The corollary which follows is of particular interest and it would be interesting to know if there is an analogous statement for compact homogeneous spaces admitting Einstein metrics.

\begin{thm}\label{thm: solv admit solsoliton reduces to irreducible components} Let $S_\mu$ be a solvable Lie group whose Lie algebra $\mathfrak s_\mu = \mathfrak s_{\mu_1} + \mathfrak s_{\mu_2}$ is a direct sum of ideals.  Then $S_\mu$ admits a non-flat solsoliton, resp. flat, metric if and only if both $S_{\mu_1}$ and $S_{\mu_2}$ admit non-flat solsoliton, resp. flat, metrics.
\end{thm}

\begin{cor}\label{cor: solv admit solsoliton reduces to irreducible components} Let $S_\mu$ be a solvable Lie group whose Lie algebra $\mathfrak s_\mu = \mathfrak s_{\mu_1} + \mathfrak s_{\mu_2}$ is a direct sum of ideals.  Then $S_\mu$ admits an Einstein metric if and only if both $S_{\mu_1}$ and $S_{\mu_2}$ admit Einstein metrics of the same sign.
\end{cor}

\begin{proof}[Proof of theorem.] We prove the case that the solsoliton is not flat.  The flat case is similar and we leave it to the reader.

\bigskip

One direction is trivial. Recall, a non-flat solsoliton with $Ric=cId+D$ satisfies $c<0$ by Theorem \ref{thm: lauret solsolitons criteria}.  If $S_{\mu_i}$ admit solsolitons satisfying $Ric_{\mu_i} = c_i Id + D_i$, then one just needs to rescale so that $c_1=c_2$.  Endow $\mathfrak s_\mu$ with the product metric, i.e. the $\mathfrak s_{\mu_i}$ are orthogonal and the restriction to $\mathfrak s_{\mu_i}$ is the aforementioned metric. Then $\mathfrak s_\mu$, with $\mu = \mu_1+\mu_2$, is a solsoliton satisfying $Ric_\mu =Ric_{\mu_1} \oplus Ric_{\mu_2} = c_1 Id + (D_1 \oplus D_2)$.

We now show the converse. Recall that $S_\mu$ admitting a solsoliton implies the orbit $GL_n\mathbb R \cdot \mu$ is distinguished by Proposition \ref{prop: admit solsoliton implies distinguished orbit}.  However, this implies the orbits $GL_{n_i}\mathbb R \cdot \mu_i \subset \wedge^2(\mathbb R^{n_i})\otimes \mathbb R^{n_i}$ are distinguished, where $n_i = \dim \mathfrak s_{\mu_i}$.  (This has been proven in \cite[Theorem 4.5]{Jablo:FinitenessTheorem-compatibleSubgroups} for nilpotent groups.  However, the proof there only uses the fact that the orbits are distinguished and works in this setting with no modifications.)

Assume now that ${\mu_i}$ are the distinguished points and write $\mathfrak s_{\mu_i} = \mathfrak a_i \oplus \mathfrak n_i$ where $\mathfrak n_i$ is the nilradical and $\mathfrak a_i$ is a reductive subalgebra (cf. \cite[Theorem 4.7]{Lauret:MomentMapVarietyLieAlgebras}).  As $\mathfrak s_{\mu_i}$ are distinguished, the nilradicals $\mathfrak n_i$ admit nilsolitons by \cite[Theorem 4.7]{Lauret:MomentMapVarietyLieAlgebras}.

Write $\mathfrak s_\mu = \mathfrak s_{\mu_1}+\mathfrak s_{\mu_2} = (\mathfrak a_1+\mathfrak a_2) + (\mathfrak n_1 + \mathfrak n_2)$.  As $\mathfrak s_\mu$ is solvable, we see that the reductive subalgebra $\mathfrak a=\mathfrak a_1+\mathfrak a_2$ is abelian and hence each $\mathfrak a_i$ is abelian.  Furthermore, for any $A\in \mathfrak a$, we see that $ad\ A : \mathfrak n \to \mathfrak n$ has no purely imaginary eigenvalues by the observations in the proof of Proposition \ref{prop: admit solsoliton implies distinguished orbit}.  Thus, the solvable groups $S_{\mu_i}$ admit solsoliton metrics by either the observations in the proof of Proposition \ref{prop: admit solsoliton implies distinguished orbit} or Theorem \ref{thm: classification of solv admitting negative Einstein}.
\end{proof}

\begin{remark*} We point out for the concerned reader that the proof of Theorem \ref{thm: classification of solv admitting negative Einstein} does not depend on the previous theorem.
\end{remark*}

\begin{proof}[Proof of corollary.]
The proof of the corollary follows immediately from the proof of the theorem and the fact that isomorphic distinguished points must be isometric.  More precisely, isomorphic distinguished points lie in the same $O(n)$-orbit, see Theorem \ref{thm: limit under grad flow}.
\end{proof}

\section{The bracket flow}\label{sec: bracket flow}
In this section we analyze the negative gradient flow of $F=||m||^2$ as the critical points of this function have geometric significance, see Theorems \ref{thm: nilsoliton vs disting point} and \ref{thm: solsolitons and distinguished orbits}.

Let $G$ be an $\theta$-stable subgroup of $GL(n,\mathbb R)$, i.e. $G$ is stable under the transpose operation.  Let $K_G$ denote the set of fixed points of $\theta(g)=(g^t)^{-1}$ (cf. Section \ref{sec: variety of Lie brackets}).  Denote the moment map of this group action by $m_G$ and consider the function $F=||m_G||^2$ with critical set $\mathfrak C^G$.  Denote the negative gradient flow of $F$ by $\varphi_t$; in the notation of Definition \ref{def: bracket flow} $\varphi_t(\mu_0)=\mu_t$.  In the following way we consider limits of this flow.

\begin{defin}\label{def: omega limit set}  The $\omega$-limit set of $\varphi _t (p) \subseteq V$ is the set $\{ q\in V \ | \ \varphi_{t_n}(p) \to q \mbox{ for some sequence } t_n \to \infty \mbox{ in } \mathbb R \}$.  We denote this set by $\omega (p)$.
\end{defin}

\begin{prop}\cite{Sjamaar:ConvexityofMomentMapping}\label{prop: unique limit point} The omega limit set $\omega (p)$ is a single point.
\end{prop}
The uniqueness of limits is a strong result and is due to the fact that  $F=||m||^2$ is real analytic, $K$-invariant, and that $\mathfrak C^G\cap \{ sphere \ of \ radius \ ||p|| \} \cap G\cdot p \subset K \cdot p$ (see theorem below).
As the limit is well-defined, we will denote it by $\omega(p) = \varphi_{\infty}(p)$.  We point out that many of the following  results  can be proven without knowing that there is a unique point in the limit set.

\begin{thm}\cite{Jablo:FinitenessTheorem-compatibleSubgroups}\label{thm: limit under grad flow} Consider $p\in \mathfrak C^G$.  Then
    \begin{quote}
    \begin{enumerate}
    \item $F(p)$ is a minimum of $F$ restricted to $G\cdot p$,
    \item $\mathfrak C^G\cap \{ sphere \ of \ radius \ ||p|| \} \cap G\cdot p \subset K \cdot p$, and
    \item $\omega (G\cdot p) \subset K_G \cdot p$, i.e. $\varphi_\infty (gp) \in K_G \cdot p$ for all $g\in G$.
    \end{enumerate}
    \end{quote}
\end{thm}

The first two statements originally appeared in \cite{Kempf-Ness} for complex representations and in \cite{Marian} for real representations.   In this way, we see that orbits containing critical points of $F=||m_G||^2$ are stable in the sense that the critical set is a global attractor of the negative gradient flow along the entire orbit.

In the setting of $GL_n\mathbb R$ acting on $V=\wedge ^2(\mathbb R^n)^*\otimes \mathbb R^n$, if $\mu \in V$ is the Lie bracket of a solvable Lie group admitting a solsoliton, then the orbit $GL_n\mathbb R \cdot \mu$ contains a critical point of the function $F=||m||^2$ (see Proposition \ref{prop: admit solsoliton implies distinguished orbit}).  We use this below to recover Einstein and solsoliton metrics.

\subsection*{Finding Einstein Metrics}
Using the above observations, we now have a procedure for recovering an Einstein, or solsoliton, metric on a solvable Lie group when it exists.

\begin{prop}\label{prop: finding Einstein via 2 curves} Let $G$ be a solvable Lie group which admits a non-flat Einstein or solsoliton metric.  The solsoliton metric may be obtained by following two consecutive curves of metrics.

Let $\ip{\cdot,\cdot}_0$ be any initial metric. The first curve $\ip{\cdot,\cdot}_t$, $t\in [0,1]$, goes from the initial metric to a so-called `distinguished metric' via a negative gradient flow.  The second curve $\ip{\cdot,\cdot}_t$, $t\in [1,2]$, joins the distinguished metric to the solsoliton metric by simply modifying the metric on the orthogonal complement $\mathfrak a$ of the nilradical $\mathfrak n$.
\end{prop}

\begin{remark*} A similar result holds for solvable Lie groups admitting flat metrics.  Here one just uses the first curve described above, cf. Proposition \ref{prop: admit solsoliton implies distinguished orbit}.
\end{remark*}

\begin{proof}
We realize this theorem by evolving the bracket instead of the metric.  Identify the metric Lie algebra $\{\mathfrak g, \ip{\ \cdot \ } \}$ with $\mathfrak s_\mu$ for some $\mu \in V$.  The first curve comes from flowing $\mu$ along the negative gradient flow of $F=||m||^2$.  This converges within the isomorphism class $GL_n\mathbb R\cdot \mu$ as the orbit is distinguished (see Proposition \ref{prop: admit solsoliton implies distinguished orbit}).  This limit is a distinguished point.

The second curve is realized by changing the metric on $\mathfrak a\times \mathfrak a$ as in the proof of Proposition \ref{prop: admit solsoliton implies distinguished orbit}:
    $$\left(\frac{-1}{c}\right) \left[ \frac{2-t}{2} \ tr \ ad\ A \circ ad\ A^t + (t-1)\ tr\ S(ad\ A)^2   \right] $$
for $t\in [1,2]$.
\end{proof}

\begin{remark}  In the event that the solsoliton is a flat Einstein metric, the second curve simply rescales the initial metric, however, when the solsoliton is not a flat Einstein metric, then the second curve consists of genuinely distinct metrics (i.e., non-homothetic metrics).
\end{remark}

\subsection*{Solitons and isometry groups}

The following is an immediate consequence of Theorem \ref{thm: limit under grad flow}.

\begin{cor}\label{cor: Kp is fixed under gradient flow} Consider $G$ a real algebraic reductive $\theta$-stable subgroup of $GL(n,\mathbb R)$ acting on $V$.  Let $G\cdot p$ be a distinguished orbit and $\varphi_\infty (p)$ as above.  Then $K_p \subset K_{\varphi_\infty(p)}$, where $K=G^\theta$ and $K_q$ is the stabilizer subgroup at $q$.
\end{cor}

\begin{proof}This follows  from the $K$-equivariance of $m_G$ and the uniqueness of integral curves of the negative gradient flow of $||m_G||^2$.  In fact, one has $K_p \subset K_{\varphi_t(p)}$ and the result follows by taking the limit.
\end{proof}

\begin{thm}\cite{Jablo:ModuliOfEinsteinAndNoneinstein}\label{thm: G dist iff ZGH dist} Consider a $\theta$-stable subgroup $G$ of $GL(n,\mathbb R)$ acting on $V$ (as in Section \ref{sec: variety of Lie brackets}).  Suppose $H$ is a $\theta$-stable group of automorphisms of $\mu \in V$.  Consider the centralizer of $H$ in $G$
    $$Z_G(H) = \{ g\in G \ | \ gh=hg \quad  \mbox{ for all } h\in H, g\in G \}$$
Then $Z_G(H)$ is reductive, $\theta$-stable and $G\cdot \mu$ is a distinguished orbit  if and only if $Z_G(H)\cdot \mu$ is a distinguished orbit.  Moreover, along the orbit $Z_G(H)\cdot \mu$, $m_G = m_{Z_G(H)}$.
\end{thm}

\begin{remark*}  In the above, there is no ambiguity as to the meaning of distinguished since $m_G = m_{Z_G(H)}$ along the subset  $Z_G(H)\cdot \mu$.

The group $H$ being a group of automorphisms means precisely that $H$ is a subgroup of the stabilizer of $GL(n,\mathbb R)$ at $\mu$, and $H$ being $\theta$-stable automatically makes $H$ reductive.
\end{remark*}

This theorem has been used to construct continuous families of Einstein nilradicals and non-Einstein nilradicals (cf. \cite{Jablo:ModuliOfEinsteinAndNoneinstein}).  We use this theorem to narrow our search for soliton metrics and to help prove the following.

\begin{thm}\label{thm: comp solv unimod solsoliton have max isom grp} Let $S$ be a completely solvable unimodular group admitting a solsoliton metric.  Let $g$ be any left-invariant metric.  Then there exists a left-invariant soliton metric $g'$ such that $Isom(S,g) \subset Isom(S,g')$, as groups.
\end{thm}

\begin{cor}\label{cor: nilsoliton have max isom grp} Let $N$ be an Einstein nilradical.  Let $g$ be any left-invariant metric.  Then there exists a left-invariant soliton metric $g'$ such that $Isom(N,g) \subset Isom(N,g')$, as groups.
\end{cor}

\begin{remark}  In this way, we see that these soliton metrics are the most symmetric (left-invariant) metric that such nilpotent and solvable groups  can admit.
\end{remark}

\begin{proof}Recall that a completely solvable unimodular Lie group $S_\mu$ admits a solsoliton metric if and only if $GL(n,\mathbb R)\cdot \mu$ is a distinguished orbit (Theorem \ref{thm: solsolitons and distinguished orbits}).  To show that such metrics have maximal isometry groups, we use an intermediate metric, a distinguished metric (i.e. critical point of $F=||m||^2$), in which the isometry group embeds and then show that this metric has the same isometry group as a particular choice of soliton metric (cf. Section \ref{sec: solitons vs distinguished orbits}).

Recall that the isometry group of a completely solvable unimodular group is its algebraic isometry group, i.e. $Isom(S_\mu) = S_\mu \rtimes (Aut(\mu) \cap  O (\ip{\ , \ })$.  Given $g\in GL(n,\mathbb R)$, $Aut(g^*\mu) = gAut(\mu)g^{-1}$ and the orthogonal group $O(\ (g^{-1})^* \ip{\ ,\ }) = g^{-1} O(\ip{\ ,\ }) g$, as $(g^{-1})^*\ip{\cdot ,\cdot }=\ip{g\cdot, g\cdot }$.

The following metric Lie algebras are isometric
    $$\{\mathbb R^n, g^*\mu ,\ip{\cdot,\cdot}\} \simeq  \{\mathbb R^n, \mu,(g^{-1})^*\ip{\cdot,\cdot}\}   $$
see Section \ref{sec: Riem Lie groups}, and the corresponding Riemmanian solvable Lie groups are isometric
    $$\{ S_{g^*\mu}, \ip{\cdot,\cdot}\} \simeq  \{ S_\mu,(g^{-1})^*\ip{\cdot,\cdot}\}   $$
At $e\in S_\mu$, the isometry group  of $\{ S_\mu,(g^{-1})^*\ip{\cdot,\cdot}\}$ has isotropy subgroup
    $$Aut(\mu) \cap  O(\ (g^{-1})^* \ip{\ ,\ }) = g^{-1} ( Aut(g^*\mu) \cap O(\ip{\ ,\ })\ )\ g     $$

\textbf{Step 1.} Let $S_\mu$ be a Riemannian solvable group which admits a solsoliton metric.  Let $H=Aut(\mu)\cap O(\ip{\ ,\ })$; this subgroup is $\theta$-stable.  By Theorem \ref{thm: solsolitons and distinguished orbits} the orbit $GL(n,\mathbb R)\cdot \mu$ is distinguished and by Theorem \ref{thm: G dist iff ZGH dist} the orbit $Z_G(H)\cdot \mu$ actually contains the limit $\mu_\infty$ of the negative gradient flow of $F=||m||^2$.  Let $g\in Z_g(H)$ be such that $g\cdot \mu = \mu_\infty$.

By Corollary \ref{cor: Kp is fixed under gradient flow}, we see that
    $$Aut(\mu) \cap O(\ip{\ ,\ }) = K_\mu \subset K_{g\cdot \mu} = Aut(g^*\mu) \cap O(\ip{\ ,\ })  $$
where $K=O(n,\mathbb R)$.  Using the fact that $g\in Z_G(H)$, we obtain
    $$Aut(\mu) \cap O(\ip{\ ,\ }) = g^{-1}(\ Aut(\mu) \cap O(\ip{\ ,\ }) \ ) \ g \subset g^{-1} ( Aut(g^*\mu) \cap O(\ip{\ ,\ })\ )\ g  = Aut(\mu) \cap  O(\ (g^{-1})^* \ip{\ ,\ }) $$
As the underlying Lie group structure of $\{S_\mu, \ip{\ ,\ } \}$ and $\{S_\mu, (g^{-1})^*\ip{\ ,\ } \}$ is the same, have have
    $$Isom(S_\mu, \ip{\ ,\ }) \subset Isom(S_\mu, (g^{-1})^*\ip{\ ,\ })$$
as Lie groups.

\textbf{Step 2.}  So far we have imbedded the isometry group of $S_\mu$ into the isometry group of a distinguished metric $S_{\mu '}$ (these are isomorphic as Lie groups).  Write $\mathfrak s_{\mu'} = \mathfrak a \ltimes \mathfrak n$.  We have already observed that the metric on $\mathfrak s_{\mu'}$ can be transformed into a solsoliton metric by simply rescaling the metric on $\mathfrak a$ and this does not change the isometry groups, see Remark \ref{rmk: comp solv unimod - dist metric vs solsoliton and isometry groups}.  This completes the proof.

\end{proof}

\subsection*{Characterization of solvable algebras admitting Einstein and solsoliton metrics}

\begin{thm}\label{thm: classification of solv admitting negative Einstein} Let $\mathfrak n$ be an Einstein nilradical and denote the algebra of derivations by $Der(\mathfrak n)$.  Let $\mathfrak a \subset Der(\mathfrak n)$ be an abelian reductive subgroup.  If no element of $\mathfrak a$ has only purely imaginary eigenvalues, then $\mathfrak s = \mathfrak a \ltimes \mathfrak n$ admits a solsoliton metric.  Moreover, every solvable algebra admitting a non-flat solsoliton metric arises this way.
\end{thm}

\begin{cor}\label{cor: classification of solv admitting negative Einstein}
If in addition to the above hypotheses,  $\mathfrak a$ contains some pre-Einstein derivation  $D$, then $\mathfrak s= \mathfrak a\ltimes \mathfrak n$ admits a negative Einstein metric. Moreover, every solvable algebra admitting a negative Einstein metric arises this way.
\end{cor}

The above characterization of solvable Lie groups admitting negative Einstein metrics is essentially a combination of the above characterization of solsolitons together with Lauret's structural results, cf. Theorem \ref{thm: lauret solsolitons criteria}. As such, we leave the proof of the corollary to the reader.  The definition of pre-Einstein derivation may be found in Definition \ref{def: pre-einstein derivation}.

\bigskip

Below we will prove that the algebras described above admit solsoliton metrics.  The fact that all solsolitons have such a rigid algebraic structure is the work of Lauret, see Theorem \ref{thm: lauret solsolitons criteria}.

\begin{proof}[Proof of Theorem \ref{thm: classification of solv admitting negative Einstein}]
Take $\mathfrak a$ as above and consider it as a subalgebra of $\mathfrak{gl}(\mathfrak n_\mu)$.  Let $A$ be the connected subgroup of $GL(\mathfrak n_\mu)$ with Lie algebra $\mathfrak a$.  Denote by $\overline A$ the Zariski closure of $A$ in $GL(\mathfrak n_\mu)$ (i.e., the smallest algebraic group containing $A$) and its Lie algebra by $\overline {\mathfrak a}$.  As $Aut(\mathfrak n_\mu)$ is an algebraic group, $\overline A \subset Aut(\mathfrak n_\mu)$.  Moreover, $\overline A$ is abelian and reductive.  The fact that $\overline A$ is abelian follows immediately from being the closure of an abelian group.  To see that this group is reductive, one can `diagonalize' $\mathfrak a$ to see that $\overline A$ is a subgroup of a torus (abelian, reductive) of $GL(\mathfrak n_\mu)$ and hence has no non-trivial nilpotent elements.

It is a classical fact that there exists  $g\in GL(n,\mathbb R)$ such that $g \overline {\mathfrak a} g^{-1}$ is $\theta$-stable since $\overline A$ is both algebraic and reductive, see \cite{Mostow:SelfAdjointGroups}. Now $\mathfrak a_0 = g\mathfrak a g^{-1}$ is a reductive, abelian subalgebra of $Der(g^*\mu)$ and
    $$\mathfrak a \ltimes \mathfrak n_\mu   \simeq \mathfrak a_0 \ltimes \mathfrak n_{g^*\mu}$$
via the isomorphism which is the identity on $\mathfrak a$ and $g$ on $\mathfrak n_\mu$.  The nilpotent Lie group $N_{g^*\mu}$ is an Einstein nilradical if and only if $N_\mu$ is so, as they are isomorphic.

We will apply Theorem \ref{thm: G dist iff ZGH dist} to the subgroup $\overline A_0=g\overline A g^{-1} \subset GL(n,\mathbb R)$ with Lie algebra $\overline{\mathfrak a_0}=g\overline a g^{-1}$.  This group is $\theta$-stable as its Lie algebra is so.  Let $\mu_0 = g^*\mu$ and consider the limit $\mu_\infty$ of the flow $\mu_t$.  The Riemannian nilpotent Lie group $N_{\mu_\infty}$ is a nilsoliton and $\mu_\infty = g'^* \mu_0$ for some $g'\in Z_{GL(n,\mathbb R)}(\overline A_0)$.  As $g'$ commutes with $\mathfrak a_0$ we see that the following solvable algebras are isomorphic
    $$\mathfrak a \ltimes \mathfrak N_{\mu} \simeq \mathfrak a_0 \ltimes \mathfrak N_{\mu_0} \simeq g' \mathfrak a_0 g'^{-1} \ltimes \mathfrak N_{g'^*\mu_0} = \mathfrak a_0 \ltimes \mathfrak N_{g'^*\mu_0} = \mathfrak a_0 \ltimes \mathfrak N_{\mu_\infty}$$
but the last metric algebra satisfies all the criteria of Theorem \ref{thm: lauret solsolitons criteria} to be a solsoliton metric Lie algebra.
\end{proof}

\subsection*{Construction of the finer subgroup $I_G(H) \subset Z_G(H) \subset G$}
In the above proofs, one can use a smaller subgroup of $Z_G(H)$ whose orbit will contain critical points of $F=||m_G||^2$.  This group will be used in Section \ref{sec: alg for nilsoliton}  to construct the algorithm that determines when a nilpotent Lie group is an Einstein nilradical.

\begin{prop}\label{prop: iG(H)}  Let $H$ be a $\theta$-stable subgroup of $Aut(\mu)$, as in Theorem \ref{thm: G dist iff ZGH dist}.  Assume $G\cap Aut(\mu) = H$, i.e. the stabilizer at $\mu$ of the group $G$ acting on $V=\wedge^2(\mathbb R^n)^*\otimes \mathbb R^n$ is $H$.

There exists a real algebraic reductive subgroup $I_G(H)$ of $G$ such that $Z_G(H)=I_G(H) (Z_G(H) \cap H)$ where $Z_G(H)\cap H$ is the stabilizer subgroup of $Z_G(H)$ at $\mu$ and the Lie algebra of $I_G(H)$ satisfies
    $$\mathfrak i_G(H) = \{ X\in \mathfrak z_G(H) \ | \ tr(XY)=0 \mbox{ for all } Y\in \mathfrak z_G(H) \cap \mathfrak h\} $$
Moreover, the orbits coincide, i.e.  $I_G(H)\cdot \mu = Z_G(H)\cdot \mu$.
\end{prop}

Here the Lie algebra $\mathfrak z _G(H)= \mathfrak z_\mathfrak g(\mathfrak h)$ of $Z_G(H)$ is the commutator of $\mathfrak h$ in $\mathfrak g$.  One can see by direct calculation that $\mathfrak i_G(H)$ is a Lie subalgebra.  We show that its corresponding Lie subgroup of $GL(n,\mathbb R)$ is an algebraic group so that we can exploit the methods of Section \ref{sec: variety of Lie brackets}.

\begin{defin} An element $X\in \mathfrak{gl}(n,\mathbb R)$ will be called algebraic if it is tangent to a real algebraic 1-parameter subgroup of $GL(n,\mathbb R)$.
More generally, a Lie subalgebra will be called algebraic if it is tangent to an algebraic subgroup of $GL(n,\mathbb R)$.
\end{defin}

An element $X\in \mathfrak g$ is called reductive if it is semisimple (over $\mathbb C$).  We observe that if $G\subset GL(n,\mathbb R)$ is any real reductive algebraic subgroup, the set of reductive algebraic elements of $\mathfrak g$ is dense.  As we are considering $G$ which are $\theta$-stable, the following bilinear form is an inner product on $\mathfrak g$
    $$\ip{X,Y} = tr(XY^t)$$
Given a $\theta$-stable element $\alpha \in \mathfrak g$ (i.e. $\alpha $ is symmetric or skew-symmetric), we define the subalgebra
    $$\mathfrak g_\alpha \ominus \alpha = \{ X\in \mathfrak g_\alpha \ | \ tr(X\alpha^t)=0\}$$
where $\mathfrak g_\alpha = \{X\in \mathfrak g \ | \ [X,\alpha]=0\}$.  Since $\alpha^t = \pm \alpha$, it follows that $\mathfrak g_\alpha \ominus \alpha$ is $\theta$-stable and  an ideal of $\mathfrak g_\alpha$.

\begin{lemma} The subalgebra $\mathfrak g_\alpha \ominus \alpha$ is an algebraic Lie subalgebra.
\end{lemma}

From this lemma, the proposition above quickly follows.  To see this, observe that $\mathfrak z_\mathfrak g (\mathfrak h) \cap \mathfrak h$ is $\theta$-stable and  decompose $\mathfrak z_\mathfrak g (\mathfrak h) \cap \mathfrak h = (\mathfrak z \cap \mathfrak h)_\mathfrak k \oplus (\mathfrak z\cap \mathfrak h)_\mathfrak p$ into its Cartan decomposition.  All the elements contained in $(\mathfrak z\cap \mathfrak h)_\mathfrak k$ and $(\mathfrak z\cap \mathfrak h)_\mathfrak p$ are algebraic reductive elements.  Now apply the above lemma to all these algebraic reductive elements and use the fact that the intersection of algebraic groups is algebraic.

\begin{proof}[Proof of lemma] The cases of $\alpha$ symmetric and skew-symmetric must be handled separately.\bigskip

\noindent
\textbf{Case: $\alpha$ symmetric.}  Every such $\alpha$ is conjugate via $O(n,\mathbb R)$ to a diagonal matrix.  As the above inner product is $Ad\ O(n,\mathbb R)$ invariant and the conjugate of an algebraic group is algebraic, we may reduce to the case that $\alpha$ is diagonal.  Also, we may reduce to the case $G=GL(n,\mathbb R)$ as the intersection of algebraic groups is algebraic.

Further more, we may assume (via conjugation by $O(n,\mathbb R)$) that the eigenvalues are weakly increasing:  $\alpha = diag\{a_1,\dots,a_1,\dots, a_k,\dots, a_k\}$.  The eigenvalues $a_i$ are rational as $\alpha$ is algebraic.  Now the subalgebra $\mathfrak g_\alpha$ consists of block diagonal matrices $\mathfrak{gl}(n_1,\mathbb R)\times \dots \times \mathfrak{gl}(n_k,\mathbb R)$.  This is clearly an algebraic Lie algebra whose Lie group $G_\alpha$ consists of the block matrices which are invertible.

The condition $X\in \mathfrak g_\alpha$ is now $\Sigma a_iX_i = 0$ where $X=blockdiag\{X_1,\dots,X_k\}$.  Write $g\in G_\alpha$ as a block diagonal matrix  $g=blockdiag\{ g_1,\dots,g_k\}$.  Then the algebraic group with Lie algebra $\mathfrak g_\alpha \ominus \alpha$ is
    $$\{g\in G_\alpha \ | \ \Pi \  det(g_i)^{qa_i} = 1 \}$$
where $q$ is the common integer such that $qa_i\in\mathbb Z$ for all $i=1,\dots, k$.\bigskip

\noindent
\textbf{Case: $\alpha$  skew-symmetric.} To prove the result in this case, we reduce to the above case and use complex algebraic groups.  We will construct a complex algebraic group whose intersection with $GL(n,\mathbb R)$ is the desired Lie group.  This Lie group will be algebraic as it is the intersection of algebraic groups.  
We refer the reader to \cite{Whitney} for an introduction to the relationship between real and complex varieties.

Observe that the above work for $\alpha$ symmetric could have been carried out over $\mathbb C$.  Consider $\mathfrak g_\alpha ^\mathbb C = \{ X\in \mathfrak{gl}(n,\mathbb C) \ | \ [X,\alpha ]=0 \}$.  Observe that $i\alpha$ has real eigenvalues (which may be assumed to be rational as above) and that $\mathfrak g_\alpha^\mathbb C = \mathfrak g_{i\alpha}^\mathbb C$ and $\mathfrak g_\alpha^\mathbb C \ominus \alpha = \mathfrak g_{i\alpha}^\mathbb C \ominus i\alpha$.  By conjugating with $U(n)\subset GL(n,\mathbb C)$, we may assume $i\alpha$ is diagonal.  Following the above work, but with complex groups instead of real, we have a complex  algebraic group over $\mathfrak g_\alpha^\mathbb C \ominus \alpha = \mathfrak g_{i\alpha}^\mathbb C \ominus i\alpha$.  Moreover, this group intersected with $GL(n,\mathbb R)$ is a real Lie group with the desired Lie algebra.  Counting dimensions, we see that the real points of this complex algebraic group are Zariksi dense and hence this Lie group is real algebraic.
\end{proof}

\section{Compact nilmanifolds}\label{sec: compact quotients}
In this section we apply the above results to compact quotients of nilpotent 
Lie groups that admit soliton metrics.

\begin{defin} Let $(M,g) = (\Gamma  \backslash N,g)$ be a compact nilmanifold where $\Gamma \subset N \subset Isom(N, g)$,  $ g$ is a left-invariant metric, and the metric on $M$ is the induced metric coming from $N$. The metric $g$ on $\Gamma \backslash N$ is called a local nilsoliton if $(N, g)$ is a nilsoliton.
\end{defin}

As in the case of Ricci solitons on nilpotent Lie groups, local nilsolitons may be characterized as critical points of a functional restricted to the set of locally $N$-invariant metrics.  In fact, these metrics are minima of the function $F(g)=\frac{\int_M tr\ Ric_g^2 \ dVol_g}{\int_M sc^2(g) \ dVol_g}$, restricted to the set of locally $N$-invariant metrics, see \cite{Jablo:RiemFunctionalOnNilpotent} for this point and more analysis on this functional.

\begin{remark} While nilsolitons are unique on a simply connected nilpotent Lie group (up to rescaling and isometry), this does not remain true for local nilsolitons on compact quotients.
\end{remark}

On compact nilmanifolds, local nilsoliton metrics are the most symmetric among all locally-left-invariant metrics (cf. Corollary \ref{cor: nilsoliton have max isom grp}).

\begin{thm}\label{thm: isom grp of local nilsoliton is maximal} Consider $M=\Gamma \backslash N$ endowed with a locally left-invariant metric $g$ where $N$ admits a nilsoliton.  Then there exists a local soliton $g'$ on $M$ such that $Isom(M,g) \subset Isom(M,g')$.
\end{thm}

\begin{proof}The proof reduces to the corresponding statement on simply-connected covers: Corollary \ref{cor: nilsoliton have max isom grp}.

Let $\phi \in Isom(M,g)$ and consider the Riemannian quotient $\pi : N \to \Gamma \backslash N$.  The map $\phi$ lifts to a diffeomorphism $\overline \phi : N\to N$ such that $\pi \circ \overline \phi = \phi \circ \pi$.  As $(\Gamma \backslash N, g)$ and $(N,g)$ are locally isometric via $\pi$ and $\phi$ is an isometry, we have $\overline \phi \in Isom(N,g)$.  Conversely, every isometry of $M$ arises from $\overline \phi \in Isom(N,g)$ satisfying the condition
    \begin{equation}\label{eqn: isometry decends to quotient}\overline \phi (\Gamma n) = \Gamma \overline \phi(n) \quad  \mbox{ for all } n\in N\end{equation}
Observe that this condition is independent of any metric data.

By Corollary \ref{cor: nilsoliton have max isom grp}, there exists a nilsoliton $g'$ on $N$ such that $\overline \phi$ is an isometry of $(N,g')$ and this choice of $g'$ holds for all $\overline \phi$.  As the above relation (\ref{eqn: isometry decends to quotient}) still holds, the diffeomorphism $\phi : M \to M$ is an isometry relative to $g'$.
\end{proof}

\begin{thm}\label{thm: existence of local nilsoliton depends only on fund grp}The existence of a local soliton depends only on the fundamental group.
\end{thm}

\begin{proof}This is a consequence of the classical fact that the fundamental group $\Gamma$ completely determines the nilpotent group $N$.  More precisely, let $\Gamma_1, \Gamma_2$ be the fundamental groups of compact nilmanifolds $\Gamma_1 \backslash N_1$ and $\Gamma_2 \backslash N_2$.  If $\phi : \Gamma_1 \to \Gamma_2$ is an isomorphism of abstract groups, there exists an isomorphism $\Phi : N_1 \to N_2$ of Lie groups such that $\phi = \Phi|_{\Gamma_1}$, see \cite[Theorem 2.11 and Corollary 2]{Raghunathan:DiscreteSubgroupsOfLieGroups}.

The claim now follows from the algorithm of Section \ref{sec: alg for nilsoliton} which shows that the existence of nilsolitons on $N_i$ is a property of the underlying Lie algebra.
\end{proof}

\begin{remark*} The above theorems on compact nilmanifolds hold for infranilmanifolds as every infranilmanifold is finitely covered by a compact nilmanifold.
\end{remark*}

\section{Stratifying the space $V$}\label{sec: stratifying V}
To refine our analysis of the Riemannian Lie groups $S_\mu$, and the function $F=||m||^2$, we stratify the space $V$.  Using this stratification, we obtain a decomposition of the automorphism group $Aut(\mu)$ which aids in the construction of algorithms to determine the existence of soliton metrics, see Lemma \ref{lemma: beta plus unipotent part is semisimple}.

Denote the critical set of $F=||m||^2$ by $\mathfrak C$.

\begin{thm}[\cite{LauretStandard,LauretWill:EinsteinSolvExistandNonexist}]\label{thm: stratification of V}
There exists a finite subset $B\subset \mathfrak g$, and for each $\beta \in B$ a $GL_n\mathbb R$-invariant smooth submanifold $S_\beta \subset V$ (a stratum), such that
    $$V\backslash \{0\} =\displaystyle \bigsqcup_{\beta\in B} S_\beta$$
This stratification satisfies  $\overline S_\beta - S_\beta = \bigsqcup_{||\beta'|| > ||\beta||} S_{\beta'}$.  Additionally, $\mathfrak C = \bigsqcup_{\beta\in B}C_\beta$ where $C_\beta \subset S_\beta$ are the critical points with critical value $M_\beta=||\beta ||^2$.

For $\mu \in S_\beta$, following conditions are satisfied:
\begin{enumerate}
    \item $\ip{[\beta,D],D} \geq 0$ for all $D\in Der(\mu)$ with equality if and only if $[\beta , D]=0$.
    \item $\beta + ||\beta||^2I$ is positive definite for all $\beta \in B$ such that $S_\beta \cap \mathcal N \not = \emptyset$.
    \item $||\beta ||^2 \leq ||m(\mu)||^2$ with equality if and only if $m(\mu)$ is conjugate to $\beta$ under $O(n)$.
    \item $tr \ \beta D =0$ for all $D\in Der(\mu)$.
    \item $\ip{\pi(\beta + ||\beta||^2I)\mu,\mu}\geq 0$ with equality if and only if $\beta +||\beta||^2I \in Der(\mu)$.
\end{enumerate}
\end{thm}

\begin{remark*}The finite subset $B\subset \mathfrak g$ consists of diagonal elements, with positive entries on the diagonal which are (weakly) increasing.
\end{remark*}

We will not reconstruct this stratification and direct the interested reader to those works above.  Instead, we describe the necessary properties below that suit our needs.  This stratification is
the real analog of  well-known stratifications in Geometric Invariant Theory over algebraically closed fields.  For representations of complex reductive groups, such stratifications coincide with a Morse theoretic stratification coming from $F=||m||^2$.  This result remains true in the setting of real representations and is an immediately consequence of Theorem \ref{thm: orbit closure in statum}.\bigskip

Given $\alpha \in diag \subset \mathfrak{gl}(n,\mathbb R)$, the diagonal matrices, we define the following groups
    \begin{eqnarray*}
    G_\alpha &=& \{ g\in GL(n,\mathbb R) \ | \ g\alpha g^{-1} = \alpha \}\\
    U_\alpha &=& \{g\in GL(n,\mathbb R)\ | \ exp(t\alpha)\ g\ exp(-t\alpha) \to e \mbox{ as t }\to -\infty \}\\
    P_\alpha &=& G_\alpha U_\alpha
    \end{eqnarray*}
$P_\alpha$ is the parabolic associated to $\alpha$ with unipotent radical $U_\alpha$ and reductive complement $G_\alpha$.  As $\alpha ^t=\alpha$, $G_\alpha$ is $\theta$-stable  and has a Cartan decomposition $G_\alpha = K_\alpha exp(\mathfrak p_\alpha)$ (cf. definition of Cartan involution in Section \ref{sec: variety of Lie brackets}).  When $\alpha \in B$, the eigenvales of $\alpha$ are weakly increasing and the group $G_\alpha$ consists of block diagonal matrices (which commute with $\alpha$) while  $U_\alpha$ is the group of lower triangular elements beneath the blocks of $G_\alpha$.

Define the subgroup $H_\alpha$ as the group with Lie algebra $\mathfrak h_\alpha = \{ X\in \mathfrak g_\alpha \ | \ tr(X\alpha)=0 \}$; this is actually an algebraic group.  In the following proposition, we maintain the notation from \cite{LauretStandard}.

\begin{defin}A point $v\in V$ is called $H_\alpha$-stable if $0\not \in \overline{H_\alpha \cdot v}$.
\end{defin}

\begin{prop}[Lauret]\label{prop: properties of stratification} Given $\beta \in B$, there exist  subsets $Z_\beta$ and $Y_\beta$ with the following properties:
    \begin{enumerate}
        \item $Y_\beta$ is $P_\beta$-invariant, $Y_\beta^{ss}=Y_\beta \cap S_\beta$ consists of $H_\beta$-semi-stable points and $S_\beta = O(n)Y_\beta^{ss}$
        \item For $y\in Y_\beta$, $\{g\in GL(n,\mathbb R) \ | \ g\cdot y \in Y_\beta \} = P_\beta$
        \item $Z_\beta = \{ v\in Y_\beta \ | \ \pi(\beta)v = ||\beta||^2 v \}$, $Z_\beta$ is $G_\beta$-invariant, $S_\beta \cap Z_\beta = Z_\beta^{ss}$ (the $H_\beta$-semi-stable points of $Z_\beta$) and $S_\beta = GL(n,\mathbb R)Z_\beta^{ss} = O(n)P_\beta Z_\beta^{ss}$
        \item The $H_\beta$ orbits intersecting $Z_\beta \cap \mathfrak C_\beta$ are all closed.
    \end{enumerate}
\end{prop}

\begin{remark*} Part ii above does not appear in \cite{LauretStandard}.  However, one can show this immediately just as in the complex case (cf. Lemma 13.4 of \cite{Kirwan}).
\end{remark*}

The following theorem and its proof have appeared in a more  general form in \cite{Heinzner-Schwarz-Stotzel-StratificationsforRealReductiveGroups}.  In our setting, a short proof is readily obtained, and so we include one for completeness.  

\begin{thm}[Heinzner-Schwarz-St{\"o}tzel]\label{thm: orbit closure in statum}Consider $\mu \in S_\beta$.  There exists a unique $GL(n,\mathbb R)$-orbit in $\overline{GL(n,\mathbb R)\cdot \mu} \cap S_\beta$ intersecting $\mathfrak C_\beta$.  The closed orbits in $S_\beta$ are precisely those intersecting $\mathfrak C_\beta$.
\end{thm}

\begin{proof}Take $\mu\in S_\beta$.  As $S_\beta = O(n)Y_\beta^{ss}$, we may assume $\mu \in Y_\beta$ and thus $GL(n,\mathbb R)\cdot \mu = O(n)P_\beta \mu$.  Since $O(n)$ is compact, we see that $\overline{GL(n,\mathbb R)\cdot \mu} = O(n) \overline{P_\beta \cdot \mu}$.

Consider any point $y\in Y_\beta$ and the curve $exp(t\beta)\cdot y$ with limit $y_{-\infty}$ as $t\to -\infty$ (this limit exists by the definition of $Y_\beta$).  Observe that $exp(t\beta) (P_\beta\cdot y) \to G_\beta \cdot y_{-\infty}$ as $t\to -\infty$.  To see this, write $g\in P_\beta$ as $g=g_1 g_2$ with $g_1\in G_\beta$ and $g_2\in U_\beta$, then $exp(t\beta) \cdot gy =  g_1 \ exp(t\beta) g_2 exp(-t\beta) \ exp(t\beta)\cdot y \to g_1 \ e \ y_{-\infty}$.

Now take $y\in \overline{GL(n,\mathbb R) \cdot \mu}\cap \mathfrak C_\beta$.  By the above theorem, there exists $k\in O(n)$ such that  $k\cdot y \in Z_\beta$.  So we may assume $k=e$ and $y\in  \overline{GL(n,\mathbb R) \cdot \mu}\cap \mathfrak C_\beta \cap Z_\beta$.  This point is fixed by $exp(t\beta)$ and we see that
    $$G_\beta \cdot y \subset \overline{G_\beta \cdot \mu_{-\infty}}$$
by applying $exp(t\beta)$ and letting $t\to -\infty$. 

As $y$ and $\mu_{-\infty}$ are both eigenvectors for $\beta$, we see that under the map $V\to \mathbb PV$, $v\mapsto [v]$, $H_\beta \cdot [y] \subset \overline{ H_\beta[\mu_{-\infty}]}$.  Now, as $H_\beta \cdot y$ is closed, the uniqueness result follows from \cite{RichSlow}.
\end{proof}

The above theorem answers Question \ref{question: gradient flow of einstein nilradical}.

\begin{cor} Let $N_{\mu_0}$ be an Einstein nilradical.  Let $\mu_\infty$ denote the limit point of the negative gradient flow of the function $F=||m||^2$ starting at $\mu_0$.  Then $\mu_\infty$ is contained in the orbit $GL_n\mathbb R \cdot \mu_0$; that is, $N_{\mu_0}$ and $N_{\mu_\infty}$ are isomorphic Lie groups.
\end{cor}

\begin{proof} This follows immediately from the fact that the limit $\mu_\infty \in S_\beta \cap \overline{GL(n,\mathbb R)\cdot \mu_0}$ and the orbit $GL(n,\mathbb R)\cdot \mu_0$ is closed in $S_\beta$.
\end{proof}

\subsection*{Automorphisms of Einstein Nilradicals}

Given that nilsolitons are precisely the critical points of $F=||m||^2$ (Theorem \ref{thm: nilsoliton vs disting point}) we have the following decomposition of $Aut(\mu)$.  The following decomposition holds more generally with $\mu$ being the critical point of $||m||^2$ and  $Aut(\mu)$ being replaced by the stabilizer subgroup of an action.  In particular, there is a similar decomposition of the automorphism group of a solvable Lie group admitting a solsoliton.

\begin{prop}\label{prop: Aut at soliton} Let $\mu\in \mathcal N$ be a soliton in the stratum $S_\beta$.  Let $G_\beta$ be the centralizer of $\beta$ in $GL_n\mathbb R$ and $U_\beta = \{g\in GL_n\mathbb R \ | \ exp(t\beta) g exp(-t\beta) \to e \ \ \mbox{ as } \ \ t\to -\infty \}$.  Then the automorphism group of $N_\mu$ decomposes as
    $$Aut(\mu) = G^\beta U^\beta = K^\beta exp(\mathfrak p^\beta) U^\beta$$
where $G^\beta = G_\beta \cap Aut$, $K^\beta = O(n)\cap G_\beta \cap Aut$, $exp(\mathfrak p^\beta) = exp(symm(n)) \cap G_\beta \cap Aut$, $U^\beta = U_\beta \cap Aut$.  
\end{prop}

\begin{proof}This result follows quickly from Proposition \ref{prop: properties of stratification}.  Let $\mu \in S_\beta$ be the nilsoliton of interest, where $S_\beta$ is the stratum defined above.  By considering $O(n)$ translates of $\mu$, we may assume that $\mu \in Z_\beta$.

Let $g\in Aut(\mu)$.  Since $g\cdot \mu = \mu \in Z_\beta \subset Y_\beta$, $g\in P_\beta$ by Part \textit{ii.} of Proposition \ref{prop: properties of stratification}, and we may write $g=g_1g_2$ where $g_1\in G_\beta$ and $g_2\in U_\beta$.  Observe that $exp(t\beta)\ g\ exp(-t\beta)$ also stabilizes $\mu$ and letting $t\to -\infty$ we see that $g_2\in Aut(\mu)$ and hence $g_1\in Aut(\mu)$.  This shows $Aut(\mu)=G^\beta U^\beta$.

Given $g\in  G^\beta$, write $g=k\ exp(X)$ where $k\in K_\beta$ and $X\in \mathfrak p_\beta$; this is possible as $G_\beta$ is stable under the transpose.  Observe, $||m (\mu)|| = ||m(g\cdot \mu)|| = ||m (exp(X)\cdot \mu)||$ and by \cite[Lemma 7.2]{Ness} we see that $exp(X)\cdot \mu = \mu$.  Thus $exp(X), \ k\in Aut(\mu)$ and the theorem is proved.
\end{proof}

There is an analogous decomposition of $Der(\mu)$ as above.  The following is presented for later use.
\begin{lemma}\label{lemma: beta plus unipotent part is semisimple} Take a nilsoliton $N_\mu$ with Einstein derivation $\beta$ and derivation algebra $Der(\mu) = \mathfrak k^\beta \oplus \mathfrak p^\beta \oplus u^\beta$.  Every element of the form
    $$ \beta + X, \quad \mbox{ with } X\in \mathfrak u^\beta$$
is semi-simple (i.e. diagonalizable).
\end{lemma}

\begin{proof}[Sketch of proof]
The proof of this fact is analogous to showing that any upper triangular matrix with non-zero distinct entries on the diagonal can be diagonalized.  One carries out similar computations in this case (as the entries of $\beta$ are all positive and $\mathfrak u^\beta$ has an appropriate block structure) to show that $\beta +X$ can be conjugated to $\beta$ via $U_\beta$.
\end{proof}

\section{Pre-Einstein derivations}\label{sec: pre-einstein derivations}

Let $\mathfrak n_\mu$ be an Einstein nilradical with Ricci tensor $Ric_\mu = cId + D$, for some $c\in \mathbb R$ and $D\in Der (\mu)$.  This derivation is semisimple with real, positive eigenvalues and satisfies the condition
    $$tr(D\psi) = -c\ tr\ \psi  $$
for any $\psi \in Der(\mu)$.  The derivation $D$ satisfying $Ric=cId + D$ is called an \textit{Einstein derivation} as its existence is necessary for solvable extensions of $\mathfrak n_\mu$ to admit an Einstein metric.  As such, we make the following definition.

\begin{defin}\label{def: pre-einstein derivation}A derivation $\phi$ of a Lie algebra $\mathfrak s_\mu$ is called pre-Einstein, if it is semisimple, with all eigenvalues real, and
    \begin{equation}\label{eqn: pre-Einstein definition}tr (\phi \psi) = tr\ \psi \mbox{ , for any }  \psi \in Der(\mu)\end{equation}
\end{defin}

The so-called Einstein derivation $D$ gives the pre-Einstein derivation $\phi = \frac{D}{-c}$, and viceversa.  Remarkably, determining the pre-Einstein derivation almost completely determines when a nilpotent Lie algebra admits a nilsoliton metric.  This derivation first appeared in \cite{Nikolayevsky:EinsteinSolvmanifoldsandPreEinsteinDerivation}.

\begin{thm}\cite{Nikolayevsky:EinsteinSolvmanifoldsandPreEinsteinDerivation}\label{thm: niko-EinsteinSolv}
    \begin{enumerate}
    \item (a) Any Lie algebra $\mathfrak s_\mu$ admits a pre-Einstein derivation $\phi_\mu$.\\
        (b) The derivation $\phi_\mu$ is determined uniquely up to automorphism of $\mathfrak s_\mu$.\\
        (c) All the eigenvalues of $\phi_\mu$ are rational numbers.
    \item Let $\mathfrak n_\mu$ be a nilpotent Lie algebra, with $\phi$ a pre-Einstein derivation. If $\mathfrak n_\mu$ is an Einstein nilradical, then its Einstein derivation is positively proportional to $\phi$.
    \end{enumerate}
\end{thm}

\begin{remark*}
Part ii. above is particularly useful as it can be difficult to determine, a priori, which stratum $\mu$ belongs to (cf. Section \ref{sec: stratifying V}).  Moreover, we will see that determining the pre-Einstein derivation reduces to solving a system of linear equations (the condition of semi-simplicity can be discarded, cf. Proposition \ref{prop: pre-einstein deriv and dropping semisimple condition}).
\end{remark*}

Let $\mathfrak n_\mu$ be a nilpotent Lie algebra with a choice of pre-Einstein derivation $\phi_\mu$.  Associated to $\phi_\mu$ we have the following subalgebra
    $$\mathfrak g_\mu = \mathfrak z(\phi_\mu) \cap ker(T) \subset \mathfrak{sl}(n,\mathbb R)$$
where $\mathfrak z(\phi_\mu)$ is the centralizer of $\phi_\mu$ and $T(A) = tr(A\phi_\mu)$.  Let $G_{\phi_\mu}\subset SL(n,\mathbb R)$ be the Lie group with algebra $\mathfrak g_{\phi_\mu}$, this is an algebraic group.

\begin{thm}\cite{Nikolayevsky:EinsteinSolvmanifoldsandPreEinsteinDerivation}\label{thm: niko- Einstein nilrad vs. closed orbit} For a nilpotent Lie algebra $\mathfrak n_\mu$  with a pre-Einstein derivation $\phi$, the following
conditions are equivalent:
    \begin{enumerate}
    \item $\mathfrak n_\mu$ is an Einstein nilradical
    \item the orbit $G_\phi \cdot \mu \subset  V$ is closed
    \end{enumerate}
\end{thm}

In this way, we see that the property of a nilpotent Lie group being an Einstein nilradical is intrinsic to its Lie algebra.  We will build on this result to obtain an algorithm which determines the condition of being an Einstein nilradical using only local data, see Section \ref{sec: alg for nilsoliton}.  To simplify our work, 
we present the following reduction.

\begin{prop}\label{prop: pre-einstein deriv and dropping semisimple condition}If $\mathfrak n_\mu$ is an Einstein nilradical, then any solution to Equation (\ref{eqn: pre-Einstein definition}) will automatically be a pre-Einstein derivation, i.e. it is automatically semi-simple with real, positive eigenvalues.
\end{prop}

\begin{remark}In this way, we see that if a solution to Equation (\ref{eqn: pre-Einstein definition}) is not semi-simple, then the nilpotent group in question is not an Einstein nilradical.
\end{remark}

\begin{proof}The proof amounts to analyzing Nikolayevki's proofs and combining those details with Lemma \ref{lemma: beta plus unipotent part is semisimple}.  For the sake of completeness,  we present Nikolayevski's proof of existence and uniqueness (up to conjugation in $Aut$) of the pre-Einstein derivation.\bigskip

First we find one pre-Einstein derivation.
Let $\mathfrak s_\mu$ be a Lie algebra and denote by $Der(\mu)$ its algebra of derivations; this is an algebraic Lie algebra (meaning it is the Lie algebra of an algebraic group).  Consider a Levi decomposition $Der(\mu) = \mathfrak s \oplus \mathfrak t \oplus \mathfrak n$ where $\mathfrak t \oplus \mathfrak n$ is the radical of $Der(\mu)$, $\mathfrak s$ is semisimple, and $\mathfrak n$ is the set of  nilpotent elements (the nilradical) of $\mathfrak t \oplus \mathfrak n$, $\mathfrak t$ is a torus (abelian subalgebra with semisimple elements), and $[\mathfrak t, \mathfrak s ]=0$.

Recall, for $\psi \in \mathfrak{gl}(n,\mathbb R)$ a semisimple endomorphism, there exist semisimple endomorphisms $\psi^\mathbb R$ and $\psi^{i\mathbb R}$ (the real and imaginary parts) which have real, resp. purely imaginary, eigenvalues such that $\psi = \psi^\mathbb R + \psi^{i\mathbb R}$ and all three endomorphisms  commute.  Moreover, the subspaces $\mathfrak t_c = \{ \psi^{i\mathbb R} \ | \ \psi \in \mathfrak t\}$ and $\mathfrak t_s = \{ \psi^{\mathbb R} \ | \ \phi\in \mathfrak t \}$ are the compact and the fully R-reducible tori (the elements of $\mathfrak t_s$ are simultaneously diagonalizable) with $\mathfrak t_s \oplus \mathfrak t_c = \mathfrak t$.

We will find a pre-Einstein derivation contained in $\mathfrak t_s$. Consider the quadratic form $b$ on $Der(\mu)$ defined by $b(\psi_1,\psi_2) = tr(\psi_1\psi_2)$.  It is a general fact that $\mathfrak n$ is in the kernel of this quadratic form, hence
    $$b(\mathfrak t, \psi) = 0 = tr(\psi)$$
for any $\psi \in \mathfrak n$.  Using the ad-invariance of $b$ (that is, $b(\psi_1, [\psi_2,\psi_3]) = b([\psi_3,\psi_1],\psi_2)$) and that $\mathfrak s=[\mathfrak s,\mathfrak s]$ is semisimple, we see that
    $$b(\mathfrak t, \psi) = 0 = tr(\psi)$$
for any $\psi \in \mathfrak s$.  Thus it suffices to solve Equation (\ref{eqn: pre-Einstein definition}) with $\phi, \psi \in \mathfrak t$.  Additionally, observe that
    $$b(\mathfrak t_s, \psi) = 0 = tr(\psi)$$
for any $\psi \in \mathfrak t_c$.  Lastly, as the quadratic form $b$ restricted to $\mathfrak t_s$ is positive definite, the existence (and uniqueness in $\mathfrak t$) follows.

To obtain the uniqueness of the pre-Einstein derivation up to conjugation in $Aut$, Nikolayevski exploits a theorem of Mostow \cite[Theorem 4.1]{Mostow:FullyReducibleSubgrpsOfAlgGrps} which says that all fully reducible subalgebras of $Der(\mu)$ are conjugate via an inner automorphism of $Der(\mu)$.  Lastly, as the center of a reducible algebra is uniquely defined, we have the desired result.\bigskip

Now we analyze this proof to study all solutions to Equation (\ref{eqn: pre-Einstein definition}).  Let $A \in \mathfrak s\oplus \mathfrak t$ be a solution to $tr(A \psi)=0$ for all $\psi\in Der(\mu)$.  We will show that $A=0$.  To see this, first assume that our Lie algebra $\mathfrak n_\mu$ is endowed with an inner product so that $\mathfrak s \oplus \mathfrak t$ is stable under the transpose operation.  This is always possible; when $\mathfrak n_\mu$ is an Einstein nilradical such a metric is explicitly given in Proposition \ref{prop: Aut at soliton}.  Using this inner product, $\psi = A^t \in Der(\mu)$ and $0=tr(A\psi)=tr(AA^t)$ implies $A=0$.

Let $\phi \in \mathfrak t$ be a pre-Einstein derivation of $\mathfrak n_\mu$, the above work shows that any solution to Equation (\ref{eqn: pre-Einstein definition}) is of the form $\phi + X$ where $X\in \mathfrak n$ (the nilpotent part of the radical of $Der(\mu)$).  And applying Lemma \ref{lemma: beta plus unipotent part is semisimple} we are finished.
\end{proof}

\section{Algorithm to determine Einstein nilradical}\label{sec: alg for nilsoliton}

In this section, we demonstrate how the existence of a nilsoliton on a nilpotent Lie group can be read off from local data.  More precisely, let $N$ be a nilpotent Lie group of interest with Lie algebra $\mathfrak n$.  To determine if $N$ admits a nilsoliton, one only needs to analyze $Der(\mathfrak n)$ and certain infinitesimal deformations of any initial left-invariant metric on $N$.

\begin{thm}\label{thm: existence of nilsoliton} The existence of a nilsoliton metric on a nilpotent Lie group $N$ is intrinsic to the underlying Lie algebra $\mathfrak n$.  More precisely, one can determine the existence of such a metric by analyzing the derivation algebra $Der(\mathfrak n)$ and infinitesimal deformations of any initial metric on $\mathfrak n$.
\end{thm}

\begin{remark} The existence of a nilsoliton being intrinsic to the Lie algebra was first shown by Nikolayevsky \cite{Nikolayevsky:EinsteinSolvmanifoldsandPreEinsteinDerivation}.  Here it was shown that the existence of such a metric is equivalent to an orbit of a particular reductive group being closed in the space of Lie brackets (see Theorem \ref{thm: niko- Einstein nilrad vs. closed orbit}).  However, it was not shown that this could be determined by measuring local data.

Before Nikolayevsky's result, it was shown by Lauret \cite{LauretNilsoliton} that the existence of such a metric is equivalent  to the full $GL_n\mathbb R$-orbit in the space of Lie brackets being so-called distinguished .  However, it was not known before the present work that this condition may be determined locally.
\end{remark}

\subsection*{Algorithm to determine if $N$ is an Einstein nilradical}

\noindent
\textbf{Step 1:} Find a solution $\phi$ to
    $$tr(\phi \psi)=tr(\psi)  \quad \mbox{ for all } \psi \in Der(\mathfrak n)$$

If the solution is $\phi$ is not semisimple (i.e. diagonalizable) with (positive) real eigenvalues  then stop, $\mathfrak n$ is not an Einstein nilradical.

If $\phi$ is semisimple with (positive) real eigenvalues, then continue; this is a  pre-Einstein derivation of $\mathfrak n$ (cf. Definition \ref{def: pre-einstein derivation} and Proposition \ref{prop: pre-einstein deriv and dropping semisimple condition}).  (Remark: positivity of the eigenvalues will be automatic if the following steps are valid.)

\bigskip

\noindent
\textbf{Step 2:} Consider the subalgebra $\mathfrak h_\mu := \mathfrak g_\phi \cap Der(\mathfrak n)$.  These are the derivations which are traceless and commute with $\phi$, see paragraph following Theorem \ref{thm: niko-EinsteinSolv}.

If $\mathfrak h_\mu$ is not reductive, then stop, $\mathfrak n$ is not an Einstein nilradical.

If $\mathfrak h_\mu$ is reductive then continue.\bigskip

To determine if this algebra is reductive: 1) compute its radical, then 2) compute the set of nilpotent elements of this radical.  The algebra is reductive if and only if the set of such nilpotent elements (in the radical) is trivial.\bigskip

\noindent
\textbf{Step 3:}  Consider the subalgebra $i_{\mathfrak g_\phi}(\mathfrak h_\mu) = \{X\in \mathfrak z_{\mathfrak g_\phi}(\mathfrak h_\mu) \ | \  tr(XY)=0 \mbox{ for all } Y\in \mathfrak z_{\mathfrak g_\phi}(\mathfrak h_\mu) \cap \mathfrak h_\mu \}$ (cf. Proposition \ref{prop: iG(H)}), where $\mathfrak z_a(b)$ denotes the centralizer of $b$ in $a$.  Let $\mathcal D$ denote the matrices of $\mathfrak{gl}_n\mathbb R$ which are diagonalizable over $\mathbb
R$; i.e., $\mathcal D =\displaystyle \bigcup_{g\in
GL_n\mathbb R} g\mathfrak t g^{-1}$, where $\mathfrak t =$ diagonal matrices of $\mathfrak
{gl}_n\mathbb R$.

Let $\mathfrak n = \mathfrak n_\mu$ corresponding to some point $\mu \in V = \wedge^2(\mathbb R^n)^*\otimes \mathbb R^n$ (see Section \ref{sec: variety of Lie brackets}).
For $X\in i_\mathfrak g(\mathfrak h_\mu) \cap \mathcal D$, write $\mu = \sum a_i\mu_i$, where $\mu_i$ is an eigen basis for $X$, i.e., $X\cdot \mu = \sum \lambda_i a_i\mu_i$.

If there is some $X\in i_{\mathfrak g_\phi}(\mathfrak h_\mu) \cap \mathcal D$ such that $\lambda_i \geq 0$ whenever $a_i\not = 0$, then $\mathfrak n$ is not an Einstein nilradical.

If for every $X$ above there exists $i$ with $\lambda_i <0$ and $a_i \not = 0$, 
then $\mathfrak n$ is an Einstein nilradical.

\begin{remark}In Step 3,

1) The identification of $\mathfrak n $ with $\mu\in V$ is made by picking a basis of the vector space.  This is tantamount to prescribing $\mathfrak n$ with an orthonormal basis, and hence, endowing $N$ with a choice of left-invariant metric.

2) The $X\cdot \mu$, with $X\in \mathfrak{gl}_n\mathbb R$, precisely represent infinitesimal deformations of the above choice of left-invariant metric.

3) The algebra $\mathfrak h_\mu$ is reductive (once getting to Step 3).  If the inner product from $\mathfrak n_\mu$ makes $\mathfrak h_\mu$  stable under the metric adjoint (and there will always be such a $\mu$ with this property), then Step 3 may be replaced by the following.
\end{remark}

\noindent
\textbf{Step 3':} Assuming $\mathfrak h_\mu$ is  stable under the adjoint relative to the inner product on $\mathfrak n_\mu$, we may reduce the collection of $X$ considered in Step 3 to those $X\in \mathfrak i_{\mathfrak g_\phi}(\mathfrak h_\mu) \cap \mathfrak p$, where $\mathfrak p =\{ Y\in \mathfrak h_\mu \ | Y^t = Y \}$. 

\begin{remark} The verification of Steps 1 and 2 above can done by a computer.  It is not immediately clear to the author if Step 3 can be adapted to be implemented by a computer.
\end{remark}

\subsection*{Proof of the algorithm above}
\textbf{Step 1:}  This is the content of Proposition \ref{prop: pre-einstein deriv and dropping semisimple condition}.\bigskip

\noindent
\textbf{Step 2:} To prove this portion of the algorithm, we will go ahead and identify $\mathfrak n$ with $\mathfrak n_\mu$, for some $\mu \in V$.  The algebra $\mathfrak h = \mathfrak g_\phi \cap Der(\mathfrak n)$ is precisely the stabilizer subalgebra of $\mathfrak g_\phi$ at $\mu$.  As we have fixed a basis of our Lie algebra, we may view $\mathfrak h \subset \mathfrak{gl}(n,\mathbb R)$.

In Theorem \ref{thm: niko- Einstein nilrad vs. closed orbit}, it was shown that $\mathfrak n_\mu$ is an Einstein nilradical if and only if $G_\phi \cdot \mu$ is closed, where $G_\phi$ is the (alegbraic) Lie group with Lie algebra $\mathfrak g_\phi$.  It is well-known that an orbit being closed implies the stabilizer subgroup is reductive, see \cite{RichSlow}.  Lastly, the stabilizer subgroup is reductive if and only if its Lie algebra $\mathfrak h$ is reductive.\bigskip

\noindent
\textbf{Step 3:} As $\mathfrak h$ and $\phi$ are reductive, there exists $g\in GL(n,\mathbb R)$ such that $g\mathfrak h g^{-1}$ and $g \phi g^{-1}$ are simultaneously $\theta$-stable, i.e. closed under the transpose operation, see \cite{Mostow:SelfAdjointGroups}.  Observe that
    \begin{quote}
    $gDer(\mu) g^{-1} = Der(g\cdot \mu)$,\\
    $g\phi g^{-1}$ is a pre-Einstein derivation of $g\cdot \mu$,\\
    $\mathfrak g_{g\phi g^{-1}} =g\mathfrak g_\phi g^{-1}$,\\
    $\mathfrak h_{g\cdot \mu} = g\mathfrak h_\mu g^{-1}$,\\
    $\mathfrak z_{\mathfrak g_{g\phi g^{-1}}} (\mathfrak h_{g\cdot \mu} ) = g \mathfrak z_{\mathfrak g_\phi} (\mathfrak h_\mu  ) g^{-1}$, and\\
    $\mathfrak i_{\mathfrak g_{g\phi g^{-1}}} (\mathfrak h_{g\cdot \mu} ) = g \mathfrak i_{\mathfrak g_\phi} (\mathfrak h_\mu  ) g^{-1}$
    \end{quote}
Moreover, $G_{g\phi g^{-1}}\cdot (g\cdot \mu) = gG_\phi g^{-1} g \mu = g (G_\phi \cdot \mu) $ is closed if and only if $G_\phi \cdot \mu$ is closed.  As such, we may reduce to the case that $\mathfrak h_\mu$ and $\phi $ are $\theta$-stable.

Since $\phi$ is $\theta$-stable, we immediately have that $\mathfrak g_\phi$ is $\theta$-stable.  Similarly, $\mathfrak z_{\mathfrak g_\phi}(\mathfrak h_\mu)$ is $\theta$-stable.  Now $\mathfrak i_{\mathfrak g_\phi} (\mathfrak h_\mu  )$ is precisely the Lie algebra of the algebraic reductive group $I_{G_\phi}(H_\mu)$ from Proposition \ref{prop: iG(H)}, where $H_\mu$ is the Lie group with Lie algebra $\mathfrak h_\mu$.

By Theorem \ref{thm: G dist iff ZGH dist}, $G_\phi \cdot \mu$ is closed if and only if $Z_{G_\phi}(H_\mu)$ is closed.  And since $I_{G_\phi}(H_\mu) \cdot \mu = Z_{G_\phi}(H_\mu)\cdot \mu$, we see that $\mathfrak n_\mu$ is an Einstein nilradical if and only if $I_{G_\phi}(H_\mu)\cdot \mu$ is closed, see Proposition \ref{prop: iG(H)} and Theorem \ref{thm: niko- Einstein nilrad vs. closed orbit}.

Observe that the stabilizer subalgebra of $\mathfrak i_{\mathfrak g_\phi}(\mathfrak h_\mu)$ at $\mu$ is trivial since it is contained in the stabilizer of $\mathfrak g_\phi$ at $\mu$ (which equals $\mathfrak h_\mu$) and $\mathfrak i_{\mathfrak g_\phi}(\mathfrak h_\mu)$ is orthogonal to $\mathfrak h_\mu$ under the inner product $\ip{A,B}=tr(AB^t)$.  Hence, the stabilizer of $I_{G_\phi}(H_\mu)$ is finite (as it is discrete and algebraic).\bigskip

As the stabilizer of $I_{G_\phi}(H_\mu)$ at $\mu$ is finite, we may apply the `Hilbert-Mumford criterion' to determine when $I_{G_\phi}(H_\mu) \cdot \mu$ is closed.  This criterion was adapted to the real setting in \cite{Birkes:OrbitsofLinAlgGrps} which states (in our setting)
    \begin{quote}$I_{G_\phi}(H_\mu)\cdot \mu$ is closed if and only if $\displaystyle \bigcup_{t\in \mathbb R} exp(tX)\cdot \mu$ is closed  for all $X\in \mathcal D \cap \mathfrak i_{\mathfrak g_\phi}(\mathfrak h_\mu)$
    \end{quote}
Roughly speaking, this criterion says that closedness of an orbit is equivalent to closedness of the orbits of all algebraic reductive 1-parameter subgroups.

To finish, we write $exp(tX)\cdot \mu = expt(tX) \sum a_i\mu_i = \sum e^{t\lambda_i} a_i\mu_i$ where $\mu_i$ is the eigenvector of $X$ above.  This set is not closed if and only if for all $i$ such that $a_i\not = 0$, either all $\lambda_i\geq 0 $ or $\lambda_i \leq 0$.  Observe that replacing $X$ with $-X$ changes the sign of the eigenvalues above and this step is proven.\bigskip

\noindent
\textbf{Step 3':} Reducing the Hilber-Mumford criterion to this smaller set of symmetric elements of $\mathfrak h_\mu$ is the content of \cite{RichSlow}.

\section{Algorithm to determine if a solvable Lie group admits a left-invariant Einstein metric}\label{sec: alg for Einstein metric}
In this section, we show that the existence of an Einstein metric on a solvable Lie group can be determined by purely local data, as in the case of nilsolitons and nilpotent Lie groups.
A similar algorithm can be written to test for the existence of solsoliton metrics.

\begin{thm}\label{thm: existence of Einstein of solvable} Let $S$ be a solvable Lie group with Lie algebra $\mathfrak s$.  The existence of a left-invariant Einstein metric on $S$ can be determined by analyzing the following: 1) adjoint action of $\mathfrak s$ on itself, 2) the commutator subalgebra $\mathfrak n = [\mathfrak s,\mathfrak s]$, 
and 3) infinitesimal deformations of any initial metric on $\mathfrak n$.
\end{thm}

\subsection*{Flat Einstein metrics}
Here we prove Theorem \ref{thm: existence of Einstein of solvable} in the case that scalar curvature is zero (such a Lie algebra is necessarily unimodular).  This amounts to showing that the solvable Lie algebra in question has the rigid algebraic structure described by Milnor, see Proposition \ref{prop: classification of solv admitting flat}.  Note, this case does not require any infinitesimal deformations of metrics on $\mathfrak n$.

Consider the adjoint action $ad\ \mathfrak s \subset Der(\mathfrak s)$ on $\mathfrak s$.  Compute the nilradical $\mathfrak n$ of $\mathfrak s$, i.e., the set of nilpotent elements.  Compute a Levi decomposition $ad\ \mathfrak s = \mathfrak T + \mathfrak N$, and let $\mathfrak t \subset \mathfrak s$ be such that $ad\ \mathfrak t=\mathfrak T$ and $\dim \mathfrak t = \dim \mathfrak T$.

\begin{lemma}If $\mathfrak s$ admits a flat metric, then $\mathfrak t$ is an abelian subalgebra, $ad\ T$ has only purely imaginary eigenvalues for $T\in \mathfrak t$, and $\dim \mathfrak t+ \dim \mathfrak n = \dim \mathfrak s$.
\end{lemma}

Proving this lemma proves the theorem as verifying the conditions on $\mathfrak t$ in the lemma amount to simply analyzing the adjoint representation of $\mathfrak s$, and  any algebra of this type admits a flat metric by Proposition \ref{prop: classification of solv admitting flat}.

\begin{proof}[Proof of lemma] Assume $\mathfrak s$ admits a flat Einstein metric.  Decompose $\mathfrak s =\mathfrak a + \mathfrak n$ where $\mathfrak n$ is the nilradical and $\mathfrak a$ is an abelian subalgebra such that $ad\ A$ has only purely imaginary eigenvalues for $A\in \mathfrak a$, cf. Proposition \ref{prop: classification of solv admitting flat}.

Observe that $ad\ \mathfrak s =ad\ \mathfrak a + ad\ \mathfrak n$ is a Levi-decomposition of $ad\ \mathfrak s$. Thus $ad\ \mathfrak a$ and $\mathfrak T = ad\ \mathfrak t$ are equal up to conjugation by $Aut(\mathfrak s)$ as they are both maximal reductive subalgebras of $ad\ \mathfrak s$ (conjugacy of such subalgebras is the main result of \cite{Mostow:FullyReducibleSubgrpsOfAlgGrps}).  As the relevant properties of $\mathfrak a$ do not change after applying an automorphism, we may assume $ad\ \mathfrak a = ad\ \mathfrak t$.  Now, the elements of $\mathfrak t$ differ from the elements of $\mathfrak a$ by only elements of the center.  Hence $\mathfrak t$ has precisely the same properties of $\mathfrak a$ and the lemma is proven.\end{proof}

\subsection*{Negative Einstein metrics}
Here we prove Theorem \ref{thm: existence of Einstein of solvable} in the case that scalar curvature is negative (such a Lie algebra is necessarily non-unimodular). Let $S$ be the solvable group in question with Lie algebra $\mathfrak s$.  Denote by  $\mathfrak n$ the commutator subalgebra $[\mathfrak s, \mathfrak s]$ of $\mathfrak s$.  (Note: when $\mathfrak s$ admits an Einstein metric, this will be the full nilradical.)

\bigskip

\noindent
\textbf{Step 1:}
    \begin{quote}
    If $\mathfrak n$ is not an Einstein nilradical, then stop, $S$ cannot admit a negative Einstein metric.\\
    If $\mathfrak n$ is an Einstein nilradical, then continue.
    \end{quote}
This step can be determined using the algorithm of Section \ref{sec: alg for nilsoliton}.\bigskip

\noindent
\textbf{Step 2:} Find a solution $\phi$ to
    $$tr(\phi \psi)=tr(\psi)  \quad \mbox{ for all } \psi \in Der(\mathfrak n)$$
within the set $ad\ \mathfrak s=\{ ad\ X \ | \ X\in \mathfrak s \} \subset Der(\mathfrak n)$.
    \begin{quote}
    If there is no non-trivial solution in this subset, or the solution is not semisimple with (positive) real eigenvalues, then stop; $S$ cannot admit a negative Einstein metric.\\
    If there is a non-trivial solution $\phi = ad\ X_\phi$, and this solution is semisimple with (positive) real eigenvalues, then continue.  Fix this choice of $X_\phi$.
    \end{quote}
This step can be verified using a computer for a given solvable Lie algebra of interest. As before, positivity of the eigenvalues will follow if the remaining steps are valid.\bigskip

\noindent
\textbf{Step 3:} Compute $\mathfrak z_{\mathfrak s}(X_\phi)=\{Y\in \mathfrak s \ | \ [Y, X_\phi ]=0 \}$
    \begin{quote}
    If $\mathfrak z_{\mathfrak s}(X_\phi)$ is not 
    abelian or $\dim \mathfrak z_{\mathfrak s}(X_\phi) + \dim \mathfrak n < \dim \mathfrak s$, then stop, $S$ does not admit a negative Einstein metric.\\
    If $\mathfrak z_{\mathfrak s}(X_\phi)$ is 
    abelian, and  $\dim \mathfrak z_{\mathfrak s}(X_\phi) + \dim \mathfrak n = \dim \mathfrak s$, then continue.
    \end{quote}
Recall, $\mathfrak z_{\mathfrak s}(X_\phi)$ is automatically reductive as $X_\phi$ is reductive, and $\mathfrak z_{\mathfrak s}(X_\phi)$ being reductive abelian implies that no element is nilpotent.  This step may be verified using a computer.  \bigskip

\noindent
\textbf{Step 4:}
    \begin{quote}
    If some element of $\mathfrak z_{\mathfrak s}(X_\phi)$ has only purely imaginary eigenvalues, then stop; $S$ does not admit a negative Einstein metric.\\
    If no element of $\mathfrak z_{\mathfrak s}(X_\phi)$ has only purely imaginary eigenvalues, then $S$ admits a negative Einstein metric.
    \end{quote}

\subsection*{Proof of the algorithm above.}

\noindent
\textbf{Step 1:} This fact is well-known, see \cite{LauretStandard}.\bigskip

\noindent
\textbf{Step 2:}  This is the content of a theorem of Nikolayevsky, see Theorem \ref{thm: niko-EinsteinSolv}, and \cite[Proposition 4.3]{Lauret:SolSolitons}.\bigskip

\noindent
\textbf{Step 3:} This follows immediately from \cite[Theorem 4.8]{Lauret:SolSolitons}.\bigskip

\noindent
\textbf{Step 4:} This is the content of Theorem \ref{thm: classification of solv admitting negative Einstein}.

\appendix
\section{Appendix: Closed orbits for general representations.}\label{sec: appendix}
The above work concerning the geometry of orbits holds in the more general framework of representations of reductive groups.  We state this result and provide only a sketch of the proof, as the proof is similar to the above case.  We do not know of this statement appearing in the literature before.

\subsection*{Closed Orbits}
\begin{thm} Let $G$ be a real reductive algebraic group acting linearly and rationally on a vector space $V$.  Determining whether an orbit $G\cdot v$ is closed in $V$ can be determined using only data from the induced representation of the Lie algebra $\mathfrak g$ at the point $v\in V$.
\end{thm}

In the following, we will only consider $G$ which is semi-simple and use the Killing form $B$ which is $Ad(G)$-invariant and symmetric.  More generally, for  a reductive group, one could use any bilinear form $B:\mathfrak g\times \mathfrak g \to \mathfrak g$ which is $Ad(G)$-invariant, symmetric, and has the property that $\{ X\in \mathfrak g \ | \ [X,\alpha]=0 \quad and \quad B(X,\alpha)=0\}$ is the Lie algebra of an algebraic group for any $\alpha \in \mathfrak g$ which is tangent to a reductive, algebraic 1-parameter subgroup.  The Killing form satisfies this condition.

\begin{proof}[Sketch of proof]  We follow the same argument as in Section 9.\bigskip

The first requirement is that $\mathfrak h = \mathfrak g_v$ be reductive.  Let $\mathfrak z_\mathfrak g(\mathfrak h)=\{X\in \mathfrak g \ | \ [X,\mathfrak h]=0\}$ denote the centralizer of $\mathfrak h$ in $\mathfrak g$. As before, consider
    $$\mathfrak i_\mathfrak g (\mathfrak h) = \{ X\in \mathfrak z_\mathfrak g(\mathfrak h) \ | \ B(X,Y)=0 \mbox{ for all } Y\in \mathfrak z_\mathfrak g(\mathfrak h) \cap \mathfrak h \}$$
These subalgebras are the Lie algebras of algebraic groups $Z_G(H)$ and $I_G(H)$, respectively, where $H=G_v$.  The orbit $G\cdot v$ is closed if and only if $Z_G(H)\cdot v = I_G(H)\cdot v$ is closed.  As $I_G(H)$ has finite stabilizer and we may apply the Hilbert-Mumford criterion.

Let $\mathcal D$ denote the matrices of $\mathfrak{gl}_n\mathbb R$ which are diagonalizable over $\mathbb R$; i.e., $\mathcal D =\displaystyle \bigcup_{g\in GL_n\mathbb R} g\mathfrak t g^{-1}$, where $\mathfrak t =$ diagonal matrices of $\mathfrak{gl}_n\mathbb R$.  Given $X\in \mathfrak i_\mathfrak g(\mathfrak h) \cap \mathcal D$, write $v=\sum a_i v_i$ where $\{v_i\}$ is an eigenvector basis of $V$ with $X\cdot v_i=\lambda_iv_i$.  The Hilbert-Mumford criterion states: $I_G(H)\cdot v$ is not closed if and only if there exists $X\in \mathfrak i_\mathfrak g(\mathfrak h)$ satisfying $\lambda_i \geq 0$ for all $i$ such that $a_i \not =0$.

In this way, we see that determining the closedness of $G\cdot v$ reduces to analyzing the stabilizer subalgebra $\mathfrak g_v$ and the representation of $\mathfrak g$ at $v$.
\end{proof}

\small
\providecommand{\bysame}{\leavevmode\hbox to3em{\hrulefill}\thinspace}
\providecommand{\MR}{\relax\ifhmode\unskip\space\fi MR }
\providecommand{\MRhref}[2]{%
  \href{http://www.ams.org/mathscinet-getitem?mr=#1}{#2}
}
\providecommand{\href}[2]{#2}


\begin{thebibliography}{BWZ04}

\bibitem[AK75]{AlekseevskiiKimelfeld:StructureOfHomogRiemSpacesWithZeroRicciCu%
rv}
D.~V. Alekseevski{\u\i} and B.~N. Kimel{'}fel{'}d, \emph{Structure of
  homogeneous {R}iemannian spaces with zero {R}icci curvature}, Functional
  Anal. Appl. \textbf{9} (1975), no.~2, 97--102. \MR{MR0402650 (53 \#6466)}

\bibitem[Bes08]{Besse:EinsteinMflds}
Arthur~L. Besse, \emph{Einstein manifolds}, Classics in Mathematics,
  Springer-Verlag, Berlin, 2008, Reprint of the 1987 edition. \MR{MR2371700
  (2008k:53084)}

\bibitem[Bir71]{Birkes:OrbitsofLinAlgGrps}
David Birkes, \emph{Orbits of linear algebraic groups}, Ann. of Math. (2)
  \textbf{93} (1971), 459--475. \MR{MR0296077 (45 \#5138)}

\bibitem[B{\"o}h04]{Bohm:HomogEinsteinMetricsAndSimplicialComplexes}
Christoph B{\"o}hm, \emph{Homogeneous {E}instein metrics and simplicial
  complexes}, J. Differential Geom. \textbf{67} (2004), no.~1, 79--165.
  \MR{2153482 (2006m:53065)}

\bibitem[BWZ04]{Bohm-Wang-Ziller:AVariationalApproachforCompactHomogEinsteinMf%
lds}
C.~B{\"o}hm, M.~Wang, and W.~Ziller, \emph{A variational approach for compact
  homogeneous {E}instein manifolds}, Geom. Funct. Anal. \textbf{14} (2004),
  no.~4, 681--733.

\bibitem[CK04]{ChowKnopf}
Bennett Chow and Dan Knopf, \emph{The {R}icci flow: an introduction},
  Mathematical Surveys and Monographs, vol. 110, American Mathematical Society,
  Providence, RI, 2004. \MR{MR2061425 (2005e:53101)}

\bibitem[DM82]{Dotti:RicciCurvUnimodularSolv}
Isabel Dotti~Miatello, \emph{Ricci curvature of left invariant metrics on
  solvable unimodular {L}ie groups}, Math. Z. \textbf{180} (1982), no.~2,
  257--263. \MR{MR661702 (84a:53044)}

\bibitem[Ebe08]{Eberlein:prescribedRicciTensor}
Patrick Eberlein, \emph{Riemannian 2-step nilmanifolds with prescribed {R}icci
  tensor}, Geometric and probabilistic structures in dynamics, Contemp. Math.,
  vol. 469, Amer. Math. Soc., Providence, RI, 2008, pp.~167--195.

\bibitem[EJ09]{EberleinJablo}
Patrick Eberlein and Michael Jablonski, \emph{Closed orbits of semisimple group
  actions and the real {H}ilbert-{M}umford function}, New developments in {L}ie
  theory and geometry, Contemp. Math., vol. 491, Amer. Math. Soc., Providence,
  RI, 2009, pp.~283--321. \MR{MR2537062}

\bibitem[GW88]{GordonWilson:IsomGrpsOfRiemSolv}
Carolyn~S. Gordon and Edward~N. Wilson, \emph{Isometry groups of {R}iemannian
  solvmanifolds}, Trans. Amer. Math. Soc. \textbf{307} (1988), no.~1, 245--269.
  \MR{MR936815 (89g:53073)}

\bibitem[Heb98]{Heber}
Jens Heber, \emph{Noncompact homogeneous {E}instein spaces}, Invent. Math.
  \textbf{133} (1998), no.~2, 279--352. \MR{MR1632782 (99d:53046)}

\bibitem[HSS08]{Heinzner-Schwarz-Stotzel-StratificationsforRealReductiveGroups}
Peter Heinzner, Gerald~W. Schwarz, and Henrik St{\"o}tzel,
  \emph{Stratifications with respect to actions of real reductive groups},
  Compos. Math. \textbf{144} (2008), no.~1, 163--185. \MR{MR2388560
  (2009a:32030)}

\bibitem[Jab08a]{Jablo:FinitenessTheorem-compatibleSubgroups}
Michael Jablonski, \emph{Detecting orbits along subvarieties via the moment
  map}, arXiv:0810.5697 [math.DG] -- to appear in M\"{u}nster Journal of Math
  (2008).

\bibitem[Jab08b]{JabloDistinguishedOrbits}
\bysame, \emph{Distinguished orbits of reductive groups}, arXiv:0806.3721v1
  [math.DG] (2008).

\bibitem[Jab09]{Jablo:ModuliOfEinsteinAndNoneinstein}
\bysame, \emph{Moduli of {E}instein and non-{E}instein nilradicals},
  arXiv:0902.1698 [math.DG] (2009).

\bibitem[Jab10]{Jablo:RiemFunctionalOnNilpotent}
\bysame, \emph{A natural {R}iemannian function on nilpotent lie groups}, in
  progress (2010).

\bibitem[Jen69]{Jensen:HomogEinsteinSpacesofDim4}
Gary~R. Jensen, \emph{Homogeneous {E}instein spaces of dimension four}, J.
  Differential Geometry \textbf{3} (1969), 309--349. \MR{MR0261487 (41 \#6100)}

\bibitem[Jen71]{Jensen}
Gary Jensen, \emph{The scalar curvature of left-invariant riemannian metrics},
  Indiana Univ. Math. J. \textbf{20} (1971), 1125--1144.

\bibitem[Kir84]{Kirwan}
Frances~Clare Kirwan, \emph{Cohomology of quotients in symplectic and algebraic
  geometry}, Mathematical Notes 31, Princeton University Press, Princeton, New
  Jersey, 1984.

\bibitem[KN78]{Kempf-Ness}
G.~Kempf and L.~Ness, \emph{The length of vectors in representation spaces},
  Springer Lecture Notes 732 (Copenhagen), Algebraic Geometry, Proceedings,
  1978, pp.~233--244.

\bibitem[Lau]{Lauret:personalcommunication}
Jorge Lauret, \emph{Personal communication}.

\bibitem[Lau01a]{LauretNilsoliton}
\bysame, \emph{Ricci soliton homogeneous nilmanifolds}, Math. Ann. \textbf{319}
  (2001), no.~4, 715--733. \MR{MR1825405 (2002k:53083)}

\bibitem[Lau01b]{Lauret:StandardEinsteinSolvAsCriticalPoints}
\bysame, \emph{Standard {E}instein solvmanifolds as critical points}, Q. J.
  Math. \textbf{52} (2001), no.~4, 463--470. \MR{MR1874492 (2002j:53048)}

\bibitem[Lau03]{Lauret:MomentMapVarietyLieAlgebras}
\bysame, \emph{On the moment map for the variety of {L}ie algebras}, J. Funct.
  Anal. \textbf{202} (2003), no.~2, 392--423.

\bibitem[Lau07]{LauretStandard}
\bysame, \emph{Einstein solvmanifolds are standard}, arXiv:math.DG/0703472 --
  to appear in Ann. of Math. (2007).

\bibitem[Lau08]{Lauret:EinsteinSolvandNilsolitonsCordobaConf2007}
\bysame, \emph{Einstein solvmanifolds and nilsolitons}, arxiv:math.DG/0806.0035
  (2008).

\bibitem[Lau10]{Lauret:SolSolitons}
\bysame, \emph{Ricci soliton solvmanifolds}, arXiv:math.DG/1002.0384 -- to
  appear in Crelle's Journal (2010).

\bibitem[LW07]{LauretWill:EinsteinSolvExistandNonexist}
Jorge Lauret and Cynthia Will, \emph{Einstein solvmanifolds: Existence and
  non-existence questions}, arXiv:math/0602502v3 [math.DG] (2007).

\bibitem[Mar01]{Marian}
Alina Marian, \emph{On the real moment map}, Math. Res. Lett. \textbf{8}
  (2001), no.~5-6, 779--788. \MR{MR1879820 (2003a:53123)}

\bibitem[Mil76]{Milnor:LeftInvMetricsonLieGroups}
John Milnor, \emph{Curvatures of left invariant metrics on {L}ie groups},
  Advances in Math. \textbf{21} (1976), no.~3, 293--329. \MR{MR0425012 (54
  \#12970)}

\bibitem[Mos55]{Mostow:SelfAdjointGroups}
G.~D. Mostow, \emph{Self-adjoint groups}, Ann. of Math. (2) \textbf{62} (1955),
  44--55. \MR{MR0069830 (16,1088a)}

\bibitem[Mos56]{Mostow:FullyReducibleSubgrpsOfAlgGrps}
\bysame, \emph{Fully reducible subgroups of algebraic groups}, Amer. J. Math.
  \textbf{78} (1956), 200--221. \MR{MR0092928 (19,1181f)}

\bibitem[Nik05]{Nikonorov:NoncompactHomogEinstein5manifolds}
Yu.~G. Nikonorov, \emph{Noncompact homogeneous {E}instein 5-manifolds}, Geom.
  Dedicata \textbf{113} (2005), 107--143. \MR{MR2171301 (2006h:53037)}

\bibitem[Nik08a]{Nikolayevsky:EinsteinSolvmanifoldsandPreEinsteinDerivation}
Y.~Nikolayevsky, \emph{Einstein solvmanifolds and the pre-{E}instein
  derivation}, (arXiv:0802.2137) to appear in Trans. Amer. Math. Soc. (2008).

\bibitem[Nik08b]{Nikolayevsky:EinsteinSolvmanifoldsattchedtoTwostepNilradicals}
\bysame, \emph{Einstein solvmanifolds attached to two-step nilradicals},
  arXiv:0805.0646v1 [math.DG] (2008).

\bibitem[Nik08c]{Nikolayevsky:EinsteinSolvwithSimpleEinsteinDerivation}
Yuri Nikolayevsky, \emph{Einstein solvmanifolds with a simple {E}instein
  derivation}, Geom. Dedicata \textbf{135} (2008), 87--102. \MR{MR2413331
  (2009f:53064)}

\bibitem[Nik08d]{Nikolayevsky:EinsteinSolvwithFreeNilradical}
\bysame, \emph{Einstein solvmanifolds with free nilradical}, Ann. Global Anal.
  Geom. \textbf{33} (2008), no.~1, 71--87. \MR{MR2369187 (2008m:53120)}

\bibitem[NM84]{Ness}
Linda Ness and David Mumford, \emph{A stratification of the null cone via the
  moment map}, American Journal of Mathematics \textbf{106} (1984), no.~6,
  1281--1329.

\bibitem[Pay10]{Payne:ExistenceofSolitononNil}
Tracy~L. Payne, \emph{The existence of soliton metrics for nilpotent {L}ie
  groups}, Geom. Dedicata \textbf{145} (2010), 71--88. \MR{MR2600946}

\bibitem[Rag72]{Raghunathan:DiscreteSubgroupsOfLieGroups}
M.~S. Raghunathan, \emph{Discrete subgroups of {L}ie groups}, Springer-Verlag,
  New York, 1972, Ergebnisse der Mathematik und ihrer Grenzgebiete, Band 68.
  \MR{MR0507234 (58 \#22394a)}

\bibitem[RS90]{RichSlow}
R.W. Richardson and P.J. Slodowy, \emph{Minimum vectors for real reductive
  algebraic groups}, J. London Math. Soc. \textbf{42} (1990), 409--429.

\bibitem[Sja98]{Sjamaar:ConvexityofMomentMapping}
Reyer Sjamaar, \emph{Convexity properties of the moment mapping re-examined},
  Adv. Math. \textbf{138} (1998), no.~1, 46--91. \MR{MR1645052 (2000a:53148)}

\bibitem[Whi57]{Whitney}
Hassler Whitney, \emph{Elementary structure of real algebraic varieties}, The
  Annals of Mathematics \textbf{66} (1957), no.~3, 545--556, 2nd Ser.

\bibitem[Wil03]{Will:Rank1EinsteinSolvOfDim7}
C.E. Will, \emph{Rank-one einstein solvmanifolds of dimension 7}, Diff. Geom.
  Appl. \textbf{19} (2003), 307--318.

\bibitem[Wil10]{Will:CurveOfNonEinsteinNilradicals}
Cynthia Will, \emph{A curve of nilpotent {L}ie algebras which are not
  {E}instein nilradicals}, Monatsh. Math. \textbf{159} (2010), no.~4, 425--437.
  \MR{MR2600907}

\bibitem[WZ86]{Wang-Ziller:ExistenceAndNonexistenceOfHomogEinstein}
McKenzie~Y. Wang and Wolfgang Ziller, \emph{Existence and nonexistence of
  homogeneous {E}instein metrics}, Invent. Math. \textbf{84} (1986), no.~1,
  177--194. \MR{830044 (87e:53081)}

\end{thebibliography}
\end{document}